\documentclass[11pt,letterpaper]{amsart}
\usepackage{amsmath,amssymb,amsfonts}
\usepackage{mathtools}
\usepackage{bbm}
\usepackage{enumerate}
\usepackage{amscd, amssymb, latexsym, amsmath, amscd}
\usepackage[all]{xy}
\usepackage{pb-diagram}
\usepackage{mathrsfs}
\usepackage{pictexwd,dcpic}
\usepackage{hyperref,url}
\usepackage{color}

\usepackage{extarrows}

\usepackage{pb-diagram}
\theoremstyle{plain}

\newcounter{theorem}[section]
\numberwithin{theorem}{section}
\numberwithin{equation}{subsection}

\newtheorem{Thm}[theorem]{Theorem}

\newtheorem{Cor}[theorem]{Corollary}
\newtheorem{Prop}[theorem]{Proposition}
\newtheorem{Lem}[theorem]{Lemma}
\newtheorem{Que}[theorem]{Question}

\theoremstyle{remark}

\newtheorem{Rmk}[theorem]{Remark}

\newcommand{\Hom}{\operatorname{Hom}}
\newcommand{\Ext}{\operatorname{Ext}}
\newcommand{\EP}{\operatorname{EP}}

\newcommand{\GL}{\operatorname{GL}}

\newcommand{\SL}{\operatorname{SL}}

\newcommand{\Ind}{\operatorname{Ind}}

\newcommand{\Jac}{\operatorname{Jac}}

\newcommand{\Irr}{\operatorname{Irr}}

\newcommand{\St}{\operatorname{St}} 
\newcommand{\Speh}{\operatorname{Speh}} 
\newcommand{\omgn}{\omega^\natural} 

\newcommand{\C}{\mathbb C}

\newcommand{\Z}{\mathbb{Z}}
\newcommand{\R}{\mathbb{R}}

\newcommand{\bm}{\begin{multline*}}
\newcommand{\tu}{\end  {multline*}}

\def\calC{\mathcal{C}}

\def\calH{\mathcal{H}}

\def\calM{\mathcal{M}}

\def\calP{\mathcal{P}}

\def\calR{\mathcal{R}}
\def\calS{\mathcal{S}}

\def\ra{\rightarrow}
\def\lra{\longrightarrow}

\long\def\zjl#1{{{\color{blue}ZJL: #1}}}

\title
{Godement--Jacquet L-function and homological theta lifting}
\author{Rui Chen, Yufeng Li, Xiaohuan Long, Chenhao Tang, Jialiang Zou} 

\address{School of Mathematical Sciences, Zhejiang University, 866 Yuhangtang Road, Hangzhou 310058, China }
\email{rchenmat@zju.edu.cn}

\address{Beijing International Center for Mathematical Research, Peking University, No.5 Yiheyuan Road, Beijing 100871, China }
\email{2401110016@pku.stu.edu.cn}

\address{School of Mathematical Sciences, University of Chinese Academy of Sciences, 19A Yuquan Road, Beijing 100049, China }
\email{longxiaohuan20@mails.ucas.ac.cn}

\address{School of Mathematics, Nanjing University, 22 Hankou Road, Nanjing 210093, China }
\email{chenhaotang@smail.nju.edu.cn}

\address{Department of Mathematics, University of Michigan, 530 Church St, Ann Arbor, MI 48109, U.S.A. }
\email{jlzou@umich.edu}

\begin{document}

\begin{abstract}
In this paper we investigate the theta lifting of type II dual pairs over a non-Archimedean local field, by combining the homological method of Adams--Prasad--Savin and the analytic method of Fang--Sun--Xue. We have three main results: 1. we determine completely the big theta lift of an irreducible representation when its Godement--Jacquet L-function is holomorphic at a critical point; 2. we compute the big theta lift of all characters, hence determine the space of eigendistributions on matrix spaces for all characters; 3. we show that the Weil representation is projective if and only if the dual pair is almost in the stable range.  
\end{abstract}

\maketitle



\section{Introduction and main results}

Determining the invariant or coinvariant of a space under an action of a group is a classical problem of fundamental interests, and it appears almost everywhere in representation theory. For example, in the classical invariant theory, one is trying to determine the invariant part of $V^{\otimes p}\otimes \left(V^*\right)^{\otimes q}$ under the diagonal action of $\GL\left(V\right)$, where $V$ is a vector space. As another example, in the local setting of the relative Langlands programme, given a spherical pair $\left(G,H\right)$ and an irreducible representation $\pi$ of $G$, one is interested in determining $H$-invariant linear functionals
$\Hom_H\left(\pi,\C\right)$
on $\pi$.

\vskip 5pt

Here comes our main problem to study in this paper. Let $F$ be a non-Archimedean local field of characteristic zero, and denote by $|\cdot|$ the absolute value of $F$. Let $M_{n,m}$ be the space of $n\times m$ matrices over $F$, and $G_n$ the group of invertible matrices in $M_{n,n}$. Denote by $\calS\left(M_{n,m}\right)$ the space of Schwartz functions (i.e. locally constant and compactly supported functions) over $M_{n,m}$, equipped with the group action of $G_n\times G_m$ given by
\begin{equation}\label{E:DefWeilRep}
    \left(\omega_{n,m}\left(g,h\right)\varphi\right)\left(X\right) = |\det g|^{-\frac{m}{2}}|\det h|^{\frac{n}{2}}\varphi\left(g^{-1}Xh\right),
\end{equation}
where $\left(g,h\right)\in G_n\times G_m$, $\varphi\in \calS\left(M_{n,m}\right)$ and $X\in M_{n,m}$. The twist $|\det g|^{-\frac{m}{2}}|\det h|^{\frac{n}{2}}$ here is to make this representation \textit{unitary} under the inner product of $L^2$-functions. 
Given any irreducible smooth representation $\pi$ of $G_n$, we would like to study the coinvariant of the tensor product of its contragredient $\pi^\vee$ with $\omega_{n,m}$ (or, the maximal $\pi$-isotypic quotient of $\omega_{n,m}$, the so called ``\textit{big theta lift}'')
\[
   \Theta_{n,m}\left(\pi\right) = \left(\omega_{n,m}\otimes\pi^\vee\right)_{G_n},
\]
as a representation of $G_m$. It is known that $\Theta_{n,m}\left(\pi\right)$ is of finite length. When $\pi=\eta$ is a character, the linear dual of $\Theta_{n,m}\left(\pi\right)$ is just the space of $\eta^{-1}$-eigendistributions on $M_{n,m}$. 

\vskip 5pt

The problem stated above is known as the \textit{type II theta lifting}. It has been studied by many people, yet it is not completely understood. The most celebrated and significant result, known as the \textit{Howe duality}, is due to M\'inguez \cite[Thm. 1]{MR2504432} in this case.

\begin{Thm}[Howe duality]\label{T:HoweDuality}
Suppose that $n\leq m$.  

\begin{enumerate}
    \item For any irreducible smooth representation $\pi$ of $G_n$, there is a unique irreducible smooth representation $\sigma$ of $G_m$, such that
    \[
        \Hom_{G_n\times G_m}\left(\omega_{n,m},\pi\boxtimes\sigma\right) \neq 0,
    \]
    and in which case the dimension of the Hom space is $1$.

    \vskip 5pt

    \item If we write $\pi$ as a Langlands quotient
    \[
        \pi = LQ\left(\tau_1\times\tau_2\times\cdots\times \tau_r\right),
    \]
    where each $\tau_i$ is an irreducible essential discrete series representation, then the representation $\sigma$ in (1) is the Langlands quotient
    \[
     \sigma = LQ \left(|\cdot|^{-\frac{m-n-1}{2}}\times\cdots\times |\cdot|^{\frac{m-n-1}{2}}\times \tau_1^\vee\times\cdots\times\tau_r^\vee\right).
    \]
\end{enumerate}
\end{Thm} 
\vskip 5pt

We need to explain our notations a bit. 

\begin{itemize}
    \item For any two smooth representations $\pi_1$ and $\pi_2$ of $G_{k_1}$ and $G_{k_2}$, we use $\pi_1\times \pi_2$ to denote the \textit{normalized} parabolic induction of $\pi_1\boxtimes\pi_2$ to $G_{k_1+k_2}$ with respect to the \textit{standard} parabolic subgroup $P_{k_1,k_2}$ whose Levi component is $G_{k_1}\times G_{k_2}$. Here ``standard'' means it contains the subgroup of upper triangular matrices. 

    \vskip 5pt

    \item For any irreducible smooth representation $\tau$ of $G_k$, there is a unique real number $s\in\R$ such that the central character of $\tau|\det|^{-s}$ is unitary. We set $e\left(\tau\right)=s$. By the \textit{Langlands classification} \cite{MR507262}, if $\tau_1,\tau_2,\cdots,\tau_r$ are irreducible essential discrete series representations such that
    \[
        e\left(\tau_1\right) \geq e\left(\tau_2\right)\geq \cdots \geq e\left(\tau_r\right),
    \]
    then the parabolic induction $\tau_1\times\tau_2\times\cdots\times\tau_r$ has a unique irreducible quotient, which is called the \textit{Langlands quotient} and will be denoted by $LQ\left(\tau_1\times\tau_2\times\cdots\times\tau_r\right)$. When the inequality does not hold, we extend the definition by setting
    \[
        LQ\left(\tau_1\times\tau_2\times\cdots\times\tau_r\right) := LQ\left(\tau'_1\times\tau'_2\times\cdots\times\tau'_r\right),
    \]
    where $\left\{\tau'_1,\cdots,\tau'_r\right\}$ is any rearrangement of $\left\{\tau_1,\cdots,\tau_r\right\}$ such that $e\left(\tau'_1\right)\geq\cdots\geq e\left(\tau'_r\right)$.
\end{itemize}
\vskip 5pt

The result of M\'inguez explicitly and completely determines the co-socle of $\Theta_{n,m}\left(\pi\right)$ (namely, the \textit{small theta lift}, which is usually denoted by $\theta_{n,m}\left(\pi\right)$). However, it does not determine $\Theta_{n,m}\left(\pi\right)$ itself, which is more natural, more complicated, and in applications one often needs to deal with it. To study the structure of $\Theta_{n,m}\left(\pi\right)$, there are two approaches by Adams--Prasad--Savin \cite{MR3753906} and Fang--Sun--Xue \cite{FSX}, which we shall describe now. 

\vskip 5pt

\subsection{Homological approach \`a la Adams--Prasad--Savin} 
\label{sub:homological_approach}

From the definition of the big theta lift $\Theta_{n,m}\left(\pi\right)$, it is easy to see that
\[
    \Theta_{n,m}\left(\pi\right)^\vee \simeq \Hom_{G_n}\left(\omega_{n,m},\pi\right)_{sm},
\]
where the subscript ``$sm$'' means taking $G_m$-smooth vectors. The idea of Adams--Prasad--Savin is to apply techniques in homological algebra to the problem: we should consider not only the Hom space, but also its derived functors $\Ext^i_{G_n}\left(\omega_{n,m},\pi\right)_{sm}$ and the Euler--Poincar\'e characteristic
\[
    \EP_{G_n}\left(\omega_{n,m},\pi\right)_{sm} = \sum_i (-1)^i \Ext^i_{G_n}\left(\omega_{n,m},\pi\right)_{sm},
\]
in the level of the Grothendieck group. Since all ``boundaries'' disappear when taking the alternating sum, it is not surprising that the Euler--Poincar\'e characteristic has a very fruitful formula \cite[Thm. 1.3]{MR3753906}.

\begin{Thm}\label{T:APS-EP-fomula}
Let $\pi$ be an irreducible smooth representation of $G_n$. Then
\[
    \EP_{G_n}\left(\omega_{n,m},\pi\right)_{sm} = \begin{cases}
        \pi\times \mathbbm{1}_{m-n} \quad &\textit{if }n\leq m;\\[10pt]

        0 \quad &\textit{otherwise}.
    \end{cases}
\]
Here $\mathbbm{1}_{m-n}$ stands for the trivial representation of $G_{m-n}$.
\end{Thm}

\vskip 5pt

With this result at hand, one can determine $\Theta_{n,m}\left(\pi\right)$ (at least in the Grothendieck group level) as long as one knows each $\Ext^i_{G_n}\left(\omega_{n,m},\pi\right)_{sm}$ for all $i>0$. In the case that $n\leq m$, it is also expected by Adams--Prasad--Savin \cite[Que. 6.3]{MR3753906} that 
\[
    \Ext^i_{G_n}\left(\omega_{n,m},\pi\right)_{sm} = 0
\]
for all $i>0$ if $\pi$ is tempered. So in this special case, one expects that $\Theta_{n,m}\left(\pi\right)\simeq \pi^\vee\times \mathbbm{1}_{m-n}$ is irreducible.

\vskip 5pt


\subsection{Analytic approach \`a la Fang--Sun--Xue} 
\label{sub:analytic_approach_}

The idea of Fang--Sun--Xue is based on the deep connection between the \textit{Godement--Jacquet L-function} and the \textit{equal rank} type II theta lifting (i.e. $n=m$ case). Recall that for a given $\varphi\in\calS\left(M_{n,n}\right)$ and a matrix coefficient $f$ of $\pi$, one can define the \textit{zeta integral}
\[
    Z\left(s,\varphi,f\right) = \int_{G_n}\varphi\left(g\right) f\left(g\right)|\det g|^{s+\frac{n-1}{2}}\, dg,
\]
which is absolutely convergent when ${\rm Re}\left(s\right)$ is sufficiently large, and admits a meromorphic continuation to the whole complex plane. Let $\varphi$ and $f$ vary. The Godement--Jacquet L-function by definition is the \textit{greatest common divisor} of all these zeta integrals. Hence one obtains a non-zero element
\[
    \frac{Z\left(s,-,-\right)}{L\left(s,\pi\right)}\in\Hom_{G_n\times G_n}\left(\omega_{n,n},\pi|\det|^{s-\frac{1}{2}}\boxtimes\pi^\vee|\det|^{-s+\frac{1}{2}}\right).
\]
By the multiplicity one result of M\'inguez (see Theorem \ref{T:HoweDuality}(1)), this is the only non-zero element up to scalar in the above Hom space.

\vskip 5pt

Now note that there is a natural open piece of $M_{n,n}$, namely the set of invertible matrices $G_n$. When ${\rm supp}\left(\varphi\right)\subset G_n$, the zeta integral is always absolutely convergent no matter how large ${\rm Re}\left(s\right)$ is. Consider the map
\[
    \Hom\left(\omega_{n,n},\pi|\det|^{s-\frac{1}{2}}\boxtimes\pi^\vee|\det|^{-s+\frac{1}{2}}\right) \lra \Hom\left(\omega^0_{n,n},\pi|\det|^{s-\frac{1}{2}}\boxtimes\pi^\vee|\det|^{-s+\frac{1}{2}}\right)
\]
induced by the natural embedding $G_n\hookrightarrow M_{n,n}$. Here $\omega^0_{n,n}$ is the subrepresentation of $\omega_{n,n}$ with the underlying vector space $\calS\left(G_n\right)$. Combining with the multiplicity one result, one can see immediately that if the L-function $L\left(s,\pi\right)$ is holomorphic at $s=1/2$, then the above map is non-zero. In this case, the boundary piece $M_{n,n}\backslash G_n$ does not contribute to $\Theta_{n,n}\left(\pi\right)$. This observation leads us to the following result \cite[Thm. 1.3]{FSX}.

\begin{Thm}
Let $\pi$ be an irreducible smooth representation of $G_n$. Suppose that at least one of $L\left(s,\pi\right)$ and $L\left(s,\pi^\vee\right)$ is holomorphic at $s=1/2$. Then $\Theta_{n,n}\left(\pi\right)\simeq \pi^\vee$ is irreducible.
\end{Thm}

\vskip 5pt

Here the L-function $L\left(s,\pi^\vee\right)$ is also involved because there is an extra symmetry of the Weil representation $\omega_{n,m}$ given by the Fourier transformation (or \textit{MVW-involution})
\[
    \omega_{n,m}^{MVW}\simeq \omega_{n,m},
\]
see Lemma \ref{L:MVW4rier}. Readers may also consult \cite[Lem. 4.3]{MR3930015} for a detailed explication.

\vskip 5pt


\subsection{Main result I: combination of two approaches} 
\label{sub:combination_of_two_approaches}

To summarize, in the equal rank case (i.e. $n=m$), we have two implications:
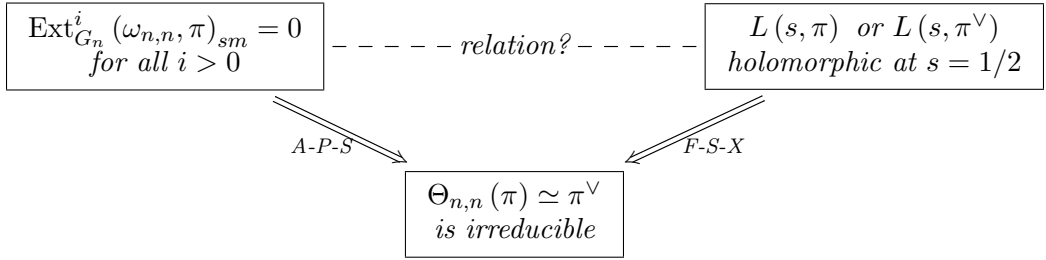
\[
\xymatrix{
\fbox{$\begin{array}{c}
    \Ext_{G_n}^i\left(\omega_{n,n},\pi\right)_{sm}=0\\
    \textit{for all }i>0
\end{array}$} \ar@{=>}_{\textit{A-P-S}}[rd] \ar@{--}[r]
&\textit{relation?} \ar@{--}[r]
& \fbox{$\begin{array}{c}
    L\left(s,\pi\right) \textit{ or } L\left(s,\pi^\vee\right)\\
    \textit{holomorphic at } s=1/2
\end{array}$}\ar@{=>}^{\textit{F-S-X}}[ld]\\
&\fbox{$\begin{array}{c}
    \Theta_{n,n}\left(\pi\right)\simeq \pi^\vee\\
    \textit{is irreducible}
\end{array}$}}
\]
It is natural to ask that whether there is any relation between the Ext-vanishing and the holomorphicity of the L-function. Moreover, one would like to go beyond the equal rank case. Our first main result answers this question.

\begin{Thm}\label{T:Main-1}
Suppose that $n\leq m$. Let $\pi$ be an irreducible smooth representation of $G_n$. If $L\left(s,\pi\right)$ or $L\left(s,\pi^\vee\right)$ is holomorphic at $s=\frac{1+m-n}{2}$, then we have
\[
    \Ext_{G_n}^i\left(\omega_{n,m},\pi\right)_{sm} = 0
\]
for all $i>0$. Moreover, we have 
\[
    \Theta_{n,m}\left(\pi\right) \simeq \begin{cases}
       \mathbbm{1}_{m-n}\times \pi^\vee \quad & \textit{if $L\left(s,\pi\right)$ is holomorphic at $s=\frac{1+m-n}{2}$};\\[10pt]

        \pi^\vee\times \mathbbm{1}_{m-n} \quad & \textit{if $L\left(s,\pi^\vee\right)$ is holomorphic at $s=\frac{1+m-n}{2}$}
    \end{cases}
\]
as finite length smooth representations. In particular, if both $L\left(s,\pi\right)$ and $L\left(s,\pi^\vee\right)$ are holomorphic at $s=\frac{1+m-n}{2}$, then $\Theta_{n,m}\left(\pi\right)$ is irreducible.
\end{Thm}

\vskip 5pt

\begin{Rmk}\label{RMK:holomorphicity-L-n=1}
It is worth noting that when $n=1$, for any character $\pi$, either $L\left(s,\pi\right)$ or $L\left(s,\pi^\vee\right)$ is holomorphic at $s=\frac{1+m-n}{2}$. In this special case the theorem recovers an old result of Xue \cite[Thm. 1.1]{MR3673813}, at least over non-Archimedean local fields. Indeed, when $n=1$ our argument also works for Archimedean local fields (under the setting of \textit{Schwartz analysis} as in \cite{MR4211018}), provided the projectivity of the Weil representation as a $G_1$-module.
\end{Rmk}

\vskip 5pt

To prove this theorem we combine the homological approach and the analytic approach. Firstly we use the \textit{Kudla's filtration} of Jacquet modules of $\omega_{n,m}$ to show that the holomorphicity of L-functions implies the Ext-vanishing. Then by adapting the zeta integral to the non equal rank case we construct a non-zero element in the space 
\[
    \Hom_{G_n\times G_m}\left(\omega_{n,m},\pi\boxtimes\theta_{n,m}\left(\pi\right)\right).
\]
Finally by mimicking Fang--Sun--Xue's argument and combining with our Ext-vanishing result we can reduce the computation of the big theta lift to the computation of the contribution of the open piece of full rank matrices, hence obtain the result. As a direct consequence, we would like to highlight the following irreducibility result of the big theta lift of tempered representations. This recovers and generalizes a very recent result of Xue on the almost equal rank type II theta lifting \cite[Thm. 3.11]{xue2023full}.

\begin{Cor}
In the context of Theorem \ref{T:Main-1}, if $\pi$ is tempered, then $\Theta_{n,m}\left(\pi\right)$ is irreducible.   
\end{Cor}

\vskip 5pt

\begin{proof}
Just notice that if $\pi$ is tempered, then both $L\left(s,\pi\right)$ and $L\left(s,\pi^\vee\right)$ are holomorphic in the right half plane ${\rm Re}\left(s\right)>0$.

\end{proof}

\vskip 5pt

Part of our argument can be generalized to the Archimedean case as well. We refer the readers to Section \ref{sub:digression_the_archimedean_case} for more details.

\vskip 5pt


\subsection{Main result II: eigendistributions on matrix spaces} 
\label{sub:main_result_ii_eigendistributions_on_matrix_spaces}

After establishing our first result Theorem \ref{T:Main-1}, the next natural question that one can ask is: in this context (i.e. $n\leq m$), what if neither $L\left(s,\pi\right)$ nor $L\left(s,\pi^\vee\right)$ is holomorphic at $s=\frac{1+m-n}{2}$? Our second result concerns this case.

\vskip 5pt

This kind of representations already show up when $n=m=2$: the trivial representation $\mathbbm{1}_2$ of $G_2$ is self-dual and its L-function has a pole at $s=1/2$. In \cite[Thm. 2.10]{xue2023full}, Xue computed $\Theta_{2,2}\left(\mathbbm{1}_2\right)$ by making use of the so called ``\textit{generalized semi-invariant distributions}'' (see \cite{MR3623236}). It quickly turns out that
\[
    \Theta_{2,2}\left(\mathbbm{1}_2\right) \simeq |\cdot|^{\frac{1}{2}}\times|\cdot|^{-\frac{1}{2}}
\]
is not irreducible, and 
\[
    \Ext_{G_2}^i\left(\omega_{2,2},\mathbbm{1}_2\right)_{sm} = \begin{cases}
        {\rm St}_{2} \quad &\textit{if }i=1;\\[10pt]

        0 \quad &\textit{if }i>1.
    \end{cases}
\]
Here ${\rm St}_{2}$ is the Steinberg representation of $G_2$. When $\min\left\{m,n\right\} >2$, there are many more such representations $\pi$. 
In this paper we treat the case when $\pi=\eta$ is a character of $G_n$. A direct computation (by using Proposition \ref{P:GJLfunction} below) shows that only when $m\leq 2n-2$, and
\[
    \eta = |{\det}_n|^{k-\frac{m}{2}}
\]
for some positive integer $1+m-n\leq k\leq n-1$, the L-functions $L\left(s,\eta\right)$ and $L\left(s,\eta^\vee\right)$ can have poles at $s=\frac{1+m-n}{2}$ simultaneously. Here the subscript ``$n$'' in $\det_n$ is employed to indicate that it is the determinant of $G_n$. We compute all Ext spaces associated to the Weil representation and these characters explicitly.

\begin{Thm}\label{T:Main-2}
Suppose that $n\leq m$. Let $\eta$ be a smooth character of $G_n$. If neither $L\left(s,\eta\right)$ nor $L\left(s,\eta^\vee\right)$ is holomorphic at $s=\frac{1+m-n}{2}$, or equivalently, if $m\leq 2n-2$ and
\[
    \eta = |\det|^{k-\frac{m}{2}}
\]
for some positive integer $1+m-n\leq k\leq n-1$, then we have
\[
    \Ext_{G_n}^i \left(\omega_{n,m}, \eta\right)_{sm} = 
    \begin{cases}
        |{\det}_{m-k}|^{\frac{k-n}{2}}\times|{\det}_k|^{\frac{n-m+k}{2}} \quad & \textit{if } i = 0; \\[10pt]

        \pi\left(n,m;k\right) \quad & \textit{if } i = 1; \\[10pt]

       0  \quad & \textit{otherwise}.
     \end{cases} 
\]
Here $\pi\left(n,m;k\right)$ is the unique irreducible quotient of the degenerate principal series representation $|{\det}_{m-k}|^{\frac{k-n}{2}}\times|{\det}_k|^{\frac{n-m+k}{2}}$. In particular, we have
\[
    \Theta_{n,m}\left(\eta\right) = |{\det}_{m-k}|^{\frac{n-k}{2}}\times|{\det}_k|^{\frac{m-n-k}{2}}
\]
is not irreducible.
\end{Thm}

\vskip 5pt

The main ingredient of the proof of this theorem is the so called ``\textit{rank filtration}'' of the Weil representation. We first compute the Ext spaces associated to each graded piece, and then ``glue'' them together by using the usual long exact sequence argument. At some point we need to show that certain map appears in the long exact sequence is actually zero, we prove this by computing some Jacquet modules. These arguments works equally well when $n>m$, leading us to the following complementary result.

\begin{Thm}\label{T:Main-2-complement}
Suppose that $n>m$, and let $\eta$ be a smooth character of $G_n$. Then
\[
    \Ext_{G_n}^i\left(\omega_{n,m},\eta\right)_{sm} = 0
\] 
for all $i\geq 0$, unless $\eta=|\det|^{k-\frac{m}{2}}$ for some integer $0\leq k\leq m$; in which case we have
\[
    \Ext_{G_n}^i\left(\omega_{n,m},\eta\right)_{sm} \simeq 
    \begin{cases}
        |\det_{m-k}|^{\frac{k-n}{2}}\times|\det_k|^{\frac{n-m+k}{2}} \quad & \textit{if } i = 0,1; \\[10pt]

       0  \quad & \textit{otherwise}.
     \end{cases}
\]
In particular, we have $\Theta_{n,m}\left(\eta\right) = |\det_{m-k}|^{\frac{n-k}{2}}\times|\det_k|^{\frac{m-n-k}{2}}$ is irreducible.   
\end{Thm}

As explicated before, the linear dual of the big theta lift of $\eta$ is exactly the space of $\eta^{-1}$-eigendistributions on $M_{n,m}$. Hence our Theorem \ref{T:Main-2} and Theorem \ref{T:Main-2-complement}, together with Theorem \ref{T:Main-1}, completely determine the space of eigendistributions on matrix spaces for all characters.

\vskip 5pt

\begin{Rmk}
In the case that $n\geq m$, there are two cases worth noting: 

\begin{itemize}
    \item if $\eta = |\det|^{\frac{m}{2}}$, then the space of $\eta^{-1}$-eigendistributions on $M_{n,m}$ is one dimensional, spanned by the \textit{$\delta$-distribution} at $0\in M_{n,m}$;

    \vskip 5pt

    \item if $\eta = |\det|^{-\frac{m}{2}}$, then the space of $\eta^{-1}$-eigendistributions on $M_{n,m}$ is also one dimensional, spanned by the \textit{Haar measure} on $M_{n,m}$.
\end{itemize}
\vskip 5pt
Moreover, these two distributions are related by the Fourier transformation.  
\end{Rmk}

\vskip 5pt

As one can see from our second main result, for a character $\eta$ of $G_n$, there exists some $i>0$ such that
\[
    \Ext_{G_n}^i\left(\omega_{n,m},\eta\right)_{sm} \neq 0
\]
if and only if both $L\left(s,\eta\right)$ and $L\left(s,\eta^\vee\right)$ have poles at $s=\frac{1+m-n}{2}$. We should mention to the readers that there is an enlightening work of Hong--Sun \cite{MR3623236}, in which they directly related the Laurent expansion of zeta integrals to generalized semi-invariant distributions. We hope to investigate the link between Ext spaces and generalized semi-invariant distributions in our future works. 

\vskip 5pt


\subsection{Main result III: projectivity of the Weil representation} 
\label{sub:main_result_iii}

In \cite[Que. 6.3]{MR3753906}, Adams--Prasad--Savin also posed the following question: given $n\leq m$, when is the Weil representation $\omega_{n,m}$ projective as a smooth $G_n$-module? We have already seen from Theorem \ref{T:Main-2} that the Weil representation is \textit{not} projective when $m\leq 2n-2$, since there does exist some characters $\eta$ of $G_n$ such that
\[
    \Ext_{G_n}^1\left(\omega_{n,m},\eta\right)_{sm} \neq 0.
\]
Our third result answers this question of Adams--Prasad--Savin.

\begin{Thm}\label{T:Main-3}
The Weil representation $\omega_{n,m}$, regarded as a smooth $G_n$-module, is projective if and only if $m\geq 2n-1$.   
\end{Thm}

\vskip 5pt

This theorem generalizes the result \cite[Lem. 6.1]{MR3753906} for $n=1$. Recall that we say the dual pair $G_n\times G_m$ is in the \textit{stable range} if $m\geq 2n$. Therefore the condition $m\geq 2n-1$ means that the dual pair $G_n\times G_m$ is ``almost'' in the stable range. The proof of this theorem has three parts. Firstly, we show that for each open compact subgroup $K$ of $G_m$, the $K$-fixed part $\omega_{n,m}^K$ of the Weil representation is a so called ``\textit{locally finite}'' representation of $G_n$, namely, for each Bernstein block the projection of $\omega_{n,m}^K$ to this block is a finitely generated module. The main tool we shall use is the rank filtration of the Weil representation $\omega_{n,m}$. Secondly, by a general result of Chan--Savin \cite[Thm. A.1]{MR3910471}, to show $\omega_{n,m}^K$ is projective, it suffices to show that
\[
    \Ext_{G_n}^i\left(\omega_{n,m},\pi\right)_{sm} = 0
\]
for all irreducible smooth representation $\pi$ of $G_n$ and all $i>0$. When $\pi$ is supercuspidal, we shall prove this by using the rank filtration; when $\pi$ is not supercuspidal, we shall prove this by using the \textit{Zelevinsky's classification} together with the Kudla's filtration. We will see in the proof that the condition $m\geq 2n-1$ guarantees that all ``boundary terms'' vanish, hence we can always reduce the computation to lower rank case. Thirdly, we pass from the ``level $K$'' part $\omega_{n,m}^K$ to $\omega_{n,m}$ by using some elementary homological algebra argument.

\vskip 5pt

Here is a direct consequence of Theorem \ref{T:Main-3} that worth noting.

\begin{Cor}
Suppose that $m\geq 2n-1$. Let $\pi$ be an irreducible smooth representation of $G_n$. Then up to semi-simplification, we have 
\[
    \Theta_{n,m}\left(\pi\right) = \pi^\vee \times \mathbbm{1}_{m-n}.
\]   
Furthermore, if $\pi$ is unitary, then $\Theta_{n,m}\left(\pi\right)$ is unitary and irreducible.
\end{Cor}

\vskip 5pt

\begin{proof}
Just combine Theorem \ref{T:Main-3} with Theorem \ref{T:APS-EP-fomula}. 

\end{proof}

\vskip 5pt

After Theorem \ref{T:Main-3} has been established, there are still many interesting questions remain to be explored. For example:

\begin{itemize}
    \item Suppose that $m\leq 2n-2$. As we already know, in this case $\omega_{n,m}$ is not projective. Then can we say something about the \textit{projective dimension} of $\omega_{n,m}$? In Proposition \ref{P:BoundExtVanishingDegree} below, we give a uniform upper bound of the projective dimension of $\omega_{n,m}^K$.

    \vskip 5pt

    \item Again suppose that $m\leq 2n-2$. Let $\pi$ be an irreducible smooth representation of $G_n$ such that
    \[
        \Ext_{G_n}^i\left(\omega_{n,m},\pi\right)_{sm} \neq 0
    \]
    for some $i>0$. Then one can ask: are these Ext spaces $\Ext_{G_n}^i\left(\omega_{n,m},\pi\right)_{sm}$ interesting objects to study? Can we upgrade them to a ``\textit{categorical}'' theta lifting?

    \vskip 5pt

    \item Can we describe the restriction of the Weil representation to the smaller member of the dual pair in some more explicit way? For instance, let $I_n\times I_m$ be the Iwahori subgroup of $G_n\times G_m$, can we describe the Iwahori fixed part $\omega_{n,m}^{I_n\times I_m}$ as a module over the Iwahori Hecke algebra?
\end{itemize}
\vskip 5pt
We wish to work on these questions in future works.

\vskip 5pt


\subsection{Speculations: Ext spaces versus local L-functions} 
\label{sub:speculation_i}

Finally in this somewhat speculative subsection, we would like to propose several questions related to the main theme of this paper, and give some examples.

\vskip 5pt

This paper begins by examining the relation between the Ext-vanishing and the holomorphicity of the L-function. In Theorem \ref{T:Main-1} we prove that the holomorphicity of the L-function implies the Ext-vanishing. The first question we would like to propose is the following: to what extent does the converse hold? To be more precise:

\begin{Que}\label{Q:ExtcontrolLfunction}
Suppose that $n\leq m$. Let $\pi$ be an irreducible representation of $G_n$. If 
\[
    \Ext_{G_n}^i\left(\omega_{n,m},\pi\right)_{sm} = 0
\]    
for all $i>0$, then can we conclude that $L\left(s,\pi\right)$ or $L\left(s,\pi^\vee\right)$ is holomorphic at $s=\frac{1+m-n}{2}$?
\end{Que}

\vskip 5pt

When $n< m$, the answer to this naive question is ``\textit{no}'', due to the following two reasons:

\begin{itemize}
    \item One can construct a counterexample by hand: let $x\in\R\backslash\frac{1}{2}\Z$, and consider the irreducible principal series representation
    \[
        \pi = |\cdot|^{\frac{1+m-n}{2}}\times\left(|\cdot|^x\times\cdots\times|\cdot|^x\right)\times |\cdot|^{-\frac{1+m-n}{2}}.
    \]
    Then the Ext-vanishing condition holds for $\pi$, but both $L\left(s,\pi\right)$ and $L\left(s,\pi^\vee\right)$ have poles at $s=\frac{1+m-n}{2}$.

    \vskip 5pt

    \item The Weil representation $\omega_{n,m}$ is projective when $m\geq 2n-1$. In this case the Ext-vanishing condition holds for any irreducible representation of $G_n$, so we can say nothing about $L\left(s,\pi\right)$ or $L\left(s,\pi^\vee\right)$.
\end{itemize}
\vskip 5pt
However, by Theorem \ref{T:Main-2}, if $\pi$ is a character, then the Ext-vanishing does imply the holomorphicity of L-functions. Hence one may ask: if we further assume that $\pi$ is \textit{unitary}, or of \textit{Arthur type}, then what is the answer to the question?

\vskip 5pt

When $n=m$, there is still some hope for the answer of the question to be ``yes''; the principal series $|\cdot|^{\frac{1}{2}}\times\left(|\cdot|^x\times\cdots\times|\cdot|^x\right)\times |\cdot|^{-\frac{1}{2}}$ is no longer a counterexample. In Section \ref{sec:speculation_ext_spaces_versus_l_functions} below, we investigate some low rank examples. We show that:

\begin{Prop}\label{P:leq3computation}
When $n=m\leq 3$, the answer to Question \ref{Q:ExtcontrolLfunction} is ``yes''. More precisely, let $\pi$ be an irreducible representation of $G_n$ such that both $L\left(s,\pi\right)$ and $L\left(s,\pi^\vee\right)$ have poles at $s=1/2$. Then there exists some $i>0$, such that 
\[
    \Ext^i_{G_n}\left(\omega_{n,n},\pi\right)_{sm}\neq 0.
\]
\end{Prop}

\vskip 5pt

Indeed, as explicated in Remark \ref{RMK:holomorphicity-L-n=1}, when $n=1$ there is no $\pi$ such that both $L\left(s,\pi\right)$ and $L\left(s,\pi^\vee\right)$ have poles at $s=1/2$; when $n=2$ there is only one such representation $\pi$, namely the trivial representation $\mathbbm{1}_2$, and we already know that 
\[
    \Ext_{G_2}^1\left(\omega_{2,2},\mathbbm{1}_2\right)_{sm} = \St_2
\]
is non-zero. So it only remains to treat the case $n=3$. A key ingredient in our computation is the zeta integral for representations that are \textit{non-irreducible}. We will show that $\Hom_{G_3}\left(\omega_{3,3},\pi\right)_{sm}$ is not irreducible if both $L\left(s,\pi\right)$ and $L\left(s,\pi^\vee\right)$ have poles at $s=1/2$. Then it follows from Theorem \ref{T:APS-EP-fomula} that the assertion in the proposition holds.

\vskip 5pt

It would also be interesting to look at other branching problems to see whether a similar phenomenon as in Theorem \ref{T:Main-1} appears. The first example comes to our mind is the \textit{Gan--Gross--Prasad} model for general linear groups. In the setting of Gan--Gross--Prasad model, given an irreducible smooth representation $\pi$ of $G_{n+1}$ and $\sigma$ of $G_n$, one would like to determine the space
\[
    \Hom_{G_n}\left(\pi,\sigma\right).
\] 
Here $G_n$ is regarded as a subgroup of $G_{n+1}$ via the map
\[
    g\longmapsto \left(\begin{array}{cc}
                            g & ~\\
                            ~ & 1
                        \end{array}\right)
\]
for $g\in G_n$. In \cite{MR3966813}, Prasad suggested to study Ext-analogues $\Ext_{G_n}^i\left(\pi,\sigma\right)$ and Euler--Poincar\'e characteristic $\EP_{G_n}\left(\pi,\sigma\right)$,
he showed that the Euler--Poincar\'e characteristic is well-defined and enjoys a very elegant formula \cite[Thm. 4.2]{MR3966813}. Therefore it will be useful to know precisely when all higher Ext spaces vanish. We would like to ask the following question.

\begin{Que}\label{Speculation-2:ExtGGPandRakinSelberg}
Let $\pi$ be an irreducible smooth representation of $G_{n+1}$ with L-parameter $\phi_\pi$, and $\sigma$ be an irreducible smooth representation of $G_{n}$ with L-parameter $\phi_\sigma$. Suppose that the L-function
\[
    L\left(s,\pi\times\sigma^\vee\right) \coloneqq L\left(s, \phi_\pi\otimes\phi_\sigma^\vee\right)
\]  
is holomorphic at $s=1/2$. Then can we conclude that for all $i>0$
\[
    \Ext_{G_n}^i\left(\pi,\sigma\right) = 0?
\]
\end{Que}

\vskip 5pt

Here are two affirmative evidences for this question. First of all, in \cite[Conj. 5.1]{MR3966813}, Prasad himself conjectured that if $\pi$ and $\sigma$ are \textit{generic}, then all higher Ext spaces vanish. This conjecture of Prasad has been proved by Chan--Savin \cite{MR4291425}. In particular, if $\pi$ and $\sigma$ are tempered, then we know that $L\left(s,\pi\times\sigma^\vee\right)$ is holomorphic at $s=1/2$, and all higher Ext spaces vanish. Secondly, we check the case $n=1$. According to a result of Chan \cite[Thm. 1.1]{MR4270667}, $\pi\,\big|_{G_1}$ is projective unless $\pi$ is a character of $G_2$. Therefore, one knows that
\[
    \Ext_{G_1}^i\left(\pi,\sigma\right) = 0
\]
for all $i>0$ unless $\pi\simeq \sigma\circ\det_2$, in which case $L\left(s,\pi\times\sigma^\vee\right)$ has a pole at $s=1/2$. 

\vskip 5pt

\begin{Rmk}
There exists some irreducible representations $\pi$ and $\sigma$ of $G_{n+1}$ and $G_n$ respectively, such that $\Ext_{G_n}^i\left(\pi,\sigma\right)=0$ for all $i>0$, but $L\left(s,\pi\times\sigma^\vee\right)$ has a pole at $s=1/2$. For example, one can take $\pi=\St_2$ and $\sigma=|\cdot|^1$. By \cite[Thm. 1.1]{MR4270667}, $\St_2\,\big|_{G_1}$ is projective. Hence  $\Ext_{G_1}^i\left(\St_2,|\cdot|^1\right)=0$ for all $i>0$, but $L\left(s,\St_2|\cdot|^{-1}\right)$ has a pole at $s=1/2$. It seems to us that unitary representations, or Arthur type representations, are more appropriate objects to consider in the context of this question. Also, this L-function $L\left(s,\pi\times\sigma^\vee\right)$ might not be optimal, one may consider the function 
\[
    L\left(s;\pi,\sigma\right)\coloneqq \frac{L\left(s,\phi_\pi\otimes\phi_\sigma^\vee\right)\cdot L\left(s,\phi_\pi^\vee\otimes\phi_\sigma\right)}{L\left(s+\frac{1}{2},\phi_\pi\otimes\phi_\pi^\vee\right)\cdot L\left(s+\frac{1}{2},\phi_\sigma\otimes\phi_\sigma^\vee\right)}
\]
defined in \cite[Thm. 3.2]{MR4190046} instead. We refer the readers to a recent paper \cite{MR4846726} for more discussions on Ext-branching laws.
\end{Rmk}

\vskip 5pt

\subsection*{Organization of the paper} 

The paper is organized as follows. In Section \ref{sec:backgrounds}, we review relevant preliminaries, including the classification of irreducible representations of general linear groups, and the local theory of Godement--Jacquet L-functions. Then in Section \ref{sec:the_weil_representation} we recall the setup of the type II theta lifting, and introduce two useful tools: the rank filtration and the Kudla's filtration. The next three sections are devoted to the proofs of our main results: the main result I, II, III will be proved in Section \ref{sec:proof_of_the_main_theorem_i}, Section \ref{sec:proof_of_the_main_theorem_ii}, Section \ref{sec:proof_of_the_main_theorem_iii} respectively. Finally in Section \ref{sec:speculation_ext_spaces_versus_l_functions} we compute some examples to provide some evidence supporting our speculations.

\vskip 5pt


\subsection*{Notations and conventions} 

We end up this introduction by setting some general notations and conventions. 

\begin{itemize}
    \item We fix a uniformizer $\varpi$ of the base field $F$, and denote by $q$ the cardinality of the residue field of $F$.

    \vskip 5pt

    \item In the rest of this paper, all representations are supposed to be smooth unless otherwise stated. So the adjective ``smooth'' will be suppressed.

    \vskip 5pt

    \item We will work in the setting of $\ell$-groups, see \cite[Def. 1]{BNote} for the definition. For an $\ell$-group $G$, we denote by $\calM\left(G\right)$ the category of all smooth representations of $G$, and by $\calM_f\left(G\right)$ the full subcategory consisting of finite length objects. The set of all irreducible representations of $G$ will be denoted by $\Irr\left(G\right)$. The Grothendieck group associated to $\Irr\left(G\right)$ will be denoted by $\calR\left(G\right)$.

    \vskip 5pt

    \item The opposite category of a category $\calC$ will be denoted by $\calC^{op}$.

    \vskip 5pt

    \item For a finite length representation $\Pi\in\calM_f\left(G\right)$, we shall write $s.s.\left(\Pi\right)$ for its semi-simplification, namely the direct sum of all its irreducible subquotients (counting multiplicities). 

    \vskip 5pt

    \item For an algebraic variety $X$ over $F$, we set $\calS\left(X\right)$ to be the space of Schwartz functions (i.e. locally constant and compactly supported functions) on its $F$-points $X\left(F\right)$.

    \vskip 5pt

    \item To distinguish different groups acting on $M_{n,m}$, from now on we shall replace the symbol ``$G$'' by ``$H$'' for all general linear groups \textit{acting on the right} of $M_{n,m}$. Namely, instead of $G_m$ (as in this introduction), from now on we shall use $H_m$ to denote the general linear group acting on the right of $M_{n,m}$, and the same change of notation also applies to all Levi subgroups of $H_m$.


\end{itemize}

\vskip 5pt

\subsection*{Acknowledgements} 
\label{ssub:acknowledgements}

This work was initiated during the summer school ``Algebra and Number Theory 2024'' held by Peking University and the Academy of Mathematics and Systems Science of the Chinese Academy of Sciences. We would like to express our sincere gratitude to Prof. Shou-Wu Zhang and Prof. Liang Xiao. Without their efforts, the summer school could not have been such a success. The first author would like to thank Prof. Binyong Sun for enlightening discussions, and Prof. Dipendra Prasad and Prof. Kei Yuen Chan for several email correspondences. 


\vskip 10pt

\section{Some backgrounds}\label{sec:backgrounds}

In this section we recall some backgrounds which will be used later in proofs of our results. 

\vskip 5pt

\subsection{Parabolic induction, Jacquet module and contragredient} 
\label{sub:parabolic_induction_jacquet_modules_and_dualities}

In the representation theory of $\ell$-groups, the functors of parabolic induction and Jacquet module play important roles. Let $G$ be a reductive $\ell$-group. For a parabolic subgroup $P$ of $G$ and a Levi decomposition $P=MN_P$, where $M$ is a Levi subgroup of $P$ and $N_P$ is the unipotent radical of $P$, recall that one can define:

\begin{itemize}
    \item the normalized parabolic induction functor
    \[
     \Ind_P^G: \calM\left(M\right) \lra \calM\left(G\right),
    \]
    by sending a representation $\left(\tau, V\right)$ of $M$ to its normalized parabolic induction
    \[
    \Ind_P^G\tau = \left\{f:G\ra V~\big|~f\left(pg\right)=\delta_P^{\frac{1}{2}}\left(p\right)\tau\left(p\right)f\left(g\right) \textit{ for all $p\in P$, $g\in G$}\right\};
    \]

    \vskip 5pt

    \item the normalized Jacquet module functor
    \[
      \Jac_P: \calM\left(G\right) \lra \calM\left(M\right),
    \]
    by sending a representation $\left(\pi, V\right)$ of $G$ to its normalized Jacquet module
    \[
        \Jac_P\pi = \delta_P^{-\frac{1}{2}}\cdot V\big/\left\langle \pi\left(n\right)v-v~\big|~n\in N_P,\, v\in V\right\rangle.
    \]
\end{itemize}
\vskip 5pt
Both of these two functors are exact, and they enjoy the Frobenius reciprocity and the second adjointness of Hom spaces, which still hold in the Ext setting by \cite[Lem. 3.2]{MR3753906}.

\begin{Lem}\label{L:IndJacAdj}
Let $G$ be a reductive $\ell$-group, $P$ a parabolic subgroup of $G$ with Levi decomposition $P=MN_P$. Denote by $\overline{P} = M N_{\overline{P}}$ the parabolic subgroup of $G$ opposite to $P$. Then for any $\pi\in\calM\left(G\right)$ and $\tau\in\calM\left(M\right)$, and any $i\geq 0$, we have:

\begin{enumerate}
    \item the Frobenius reciprocity:
    \[
    \Ext_G^i\left(\pi, \Ind_P^G\tau\right) \simeq \Ext_M^i\left(\Jac_P\pi, \tau\right);
    \]

    \vskip 5pt

    \item the second adjointness:
    \[
    \Ext_G^i\left(\Ind_P^G\tau, \pi\right) \simeq \Ext_M^i\left(\tau, \Jac_{\overline{P}}\pi\right).
    \]
\end{enumerate}
\end{Lem}
\vskip 5pt

In our later arguments, a Levi subgroup is usually a product of two (or more) groups. The following K\"unneth formula \cite[Thm. 3.5]{prasad2023homological} will be useful.

\begin{Thm}\label{T:Kunneth}
Let $G$ and $H$ be two reductive $\ell$-groups. Let $\pi_1$, $\pi_2$ be two representations of $G$, and $\sigma_1$, $\sigma_2$ be two representations of $H$. Assume that one of the following two conditions hold:
\begin{enumerate}
    \item both the representations $\pi_1$ and $\pi_2$ of $G$ have finite lengths; or

    \vskip 5pt

    \item the representation $\pi_1$ of $G$ and $\sigma_1$ of $H$ have finite lengths.
\end{enumerate}
Then we have
\[
    \Ext_{G\times H}^i\left(\pi_1\boxtimes\sigma_1,\pi_2\boxtimes\sigma_2\right) \simeq \bigoplus_{i=j+k}\Ext^j_G\left(\pi_1,\pi_2\right)\otimes\Ext_H^k\left(\sigma_1,\sigma_2\right).
\]
\end{Thm}
\vskip 5pt

Another important concept is the contragredient representation. This can be also regarded as an exact \textit{contravariant} functor
\[
  (-)^\vee: \calM\left(G\right)^{op} \lra \calM\left(G\right)  
\]
by sending a representation $\pi$ of $G$ to its contragredient $\pi^\vee$. 
For any representation $\pi$ of $G$ and $\sigma$ of $M$, where $M$ is a Levi subgroup contained in a parabolic subgroup $P$ of $G$, we have 
\[
    \Ind_P^G\left(\sigma^\vee\right) \simeq \left(\Ind_P^G \sigma\right)^\vee, \quad
    \Jac_{\overline{P}}\left(\pi^\vee\right) \simeq \left(\Jac_P\pi\right)^\vee.
\]
Later we will also introduce another related \textit{covariant} functor for general linear groups.

\vskip 5pt


\subsection{Hom functor with extra symmetries} 
\label{sub:hom_functor_with_extra_symmetry}

Given two representations $\omega$ and $\pi$ of $G$, one can consider the Ext-spaces 
\[
    \Ext^i_G\left(\omega, \pi\right),\quad  i\geq 0.
\] 
If $\omega$ is moreover a representation of $G\times H$ for another group $H$, then by the functoriality, these Ext spaces will be also equipped with an action of $H$, possibly not smooth. By taking the smooth part, we get bi-functors
\[
  \Ext^i_G\left(-,-\right)_{sm}:\calM\left(G\times H\right)^{op}\times \calM\left(G\right) \lra \calM\left(H\right),\quad i\geq 0.
\]

\begin{Lem}
For each $i\geq 1$, the $i$-th derived functor of $\Hom_G\left(-,-\right)_{sm}=\Ext^0_G\left(-,-\right)_{sm}$ coincides with $\Ext^i_G\left(-,-\right)_{sm}$. More precisely, for any $\omega\in\calM\left(G\times H\right)$ and any $\pi\in\calM\left(G\right)$, we have
\[
    R^i\Hom_G\left(\omega,-\right)_{sm}\left(\pi\right)\simeq R^i\Hom_G\left(-,\pi\right)_{sm}\left(\omega\right) \simeq \Ext^i_G\left(\omega, \pi\right)_{sm}.
\]
Here $R^i\Hom_G\left(\omega,-\right)_{sm}$ means the $i$-th right derived functor of $\Hom_G\left(\omega,-\right)_{sm}$, and likewise $R^i\Hom_G\left(-,\pi\right)_{sm}$ means the $i$-th right derived functor of $\Hom_G\left(-,\pi\right)_{sm}$.
\end{Lem}

\vskip 5pt

\begin{proof}
The essence of the proof is to deal with the procedure of taking smooth vectors.
Let $\calH$ be the Hecke algebra of $H$, i.e. the convolution algebra of smooth compactly supported distributions on $H$. Recall that any smooth representation of $H$ can be regarded as a non-degenerate module over $\calH$. Hence for any $\omega\in\calM\left(G\times H\right)$ and any $\pi\in\calM\left(G\right)$, the Hom space $\Hom_G\left(\omega,\pi\right)$ is equipped with an $\calH$-module structure given by
\[
    \left(\xi,F\right) \longmapsto F\left(\check{\xi}\cdot-\right)
\]
for $\xi\in\calH$ and $F\in\Hom_G\left(\omega,\pi\right)$. Here $\check{\xi}$ is the push-forward of $\xi$ along the inverse map $h\mapsto h^{-1}$ of $H$.
Consider a new category $\calM\left(\calH\right)$ of \textit{all} modules of $\calH$, not necessarily non-degenerate. Then as the discussion above, the bi-functor $\Hom_G\left(-,-\right)_{sm}$ can be decomposed as the composition
\[
    (-)_{sm}\circ\Hom_G\left(-,-\right):\calM\left(G\times H\right)^{op}\times \calM\left(G\right) \lra \calM\left(\calH\right)\lra \calM\left(H\right),
\]
where $\Hom_G\left(-,-\right)$ is the usual Hom functor, except for that we restrict the first variable to $G\times H$-modules (so that the resulting vector space is an $\calH$-module); and
\[
    (-)_{sm}: \calM\left(\calH\right)\lra \calM\left(H\right)
\]
is the functor of taking smooth vectors, which is exact by \cite[I.1.2 Prop.]{MR2567785}.

\vskip 5pt

Now this lemma follows from this decomposition easily. If $\omega$ is projective as a representation of $G\times H$, then by \cite[Prop. 2.3]{MR3966813} it is also projective as a representation of $G$. Hence the functor
\[
    \Hom_G\left(\omega,-\right)_{sm} = (-)_{sm}\circ\Hom_G\left(\omega,-\right): \calM\left(G\right) \lra \calM\left(\calH\right) \lra \calM\left(H\right)
\]
is exact. Likewise, if $\pi$ is an injective representation of $G$, the functor
\[
   \Hom_G\left(-,\pi\right)_{sm} =  (-)_{sm}\circ \Hom_G\left(-,\pi\right):\calM\left(G\times H\right)^{op}\lra \calM\left(\calH\right) \lra \calM\left(H\right)
\]
is also exact. Therefore we have
\[
    R^i\Hom_G\left(\omega,-\right)_{sm}\left(\pi\right)\simeq R^i\Hom_G\left(-,\pi\right)_{sm}\left(\omega\right) \simeq \Ext^i_G\left(\omega, \pi\right)_{sm}
\]
as desired.

\end{proof}
\vskip 5pt


\subsection{Representations of general linear groups} 
\label{sub:representations_of_general_linear_groups}

In this subsection we collect some results on representations of $p$-adic general linear groups. Let
\[
    \calR= \bigoplus_{n\geq 0} \calR\left(G_n\right).
\]
Recall that for representations $\pi_1$ and $\pi_2$ of $G_{k_1}$ and $G_{k_2}$, one can produce a representation $\pi_1\times \pi_2$ of $G_{k_1+k_2}$ using the normalized parabolic induction. This induces a $\Z$-graded ring structure on $\calR$. We denote this multiplication by
\[
    m: \calR\otimes \calR \lra \calR.
\]
On the other hand, for each $0\leq k \leq n$, there is a unique standard maximal parabolic subgroup $P_{k}$ of 
$G_n$, with Levi component $L_k\simeq G_k\times G_{n-k}$. Taking Jacquet modules along these parabolic subgroups induces a \textit{co-multiplication}, that is, a ring homomorphism
\[
    m^*: \calR \lra \calR\otimes \calR,
\]
given by sending $\pi\in \calR\left(G_n\right)$ to
\[
    m^*\left(\pi\right) = \sum_{0\leq k \leq n} s.s.\Jac_{P_k}\pi.
\]
Then $m$ and $m^*$ make $\calR$ a \textit{graded Hopf algebra}. This fact is useful to our later computations.

\vskip 5pt

Next we recall the notion of \textit{segment}, the building block of representations of general linear groups. A segment $\Delta = \left[x,y\right]_\rho$ is a set of supercuspidal representations of the form
\[
    \rho|\det|^x, \rho|\det|^{x-1}, \cdots, \rho|\det|^y,
\]
where $\rho$ is a \textit{unitary} supercuspidal representation of a general linear group $G_k$, and $x,y\in\R$ with $x-y\in\Z_{\geq 0}$. We note that, although usually the definition of segment doesn't require $\rho$ to be unitary, assuming the unitarity causes no loss of generality. Indeed, for any irreducible supercuspidal representation $\rho'$, its unramified twist $\rho'|\det|^{-e\left(\rho'\right)}$ is unitary. To such a segment one can associate two irreducible representations to it as follows. Consider the parabolic induction of these supercuspidal representations
\[
    \rho|\det|^x\times \rho|\det|^{x-1}\times \cdots\times\rho|\det|^y.
\]
It has a unique subrepresentation (resp. quotient), called the \textit{Steinberg} (resp. \textit{Speh}) representation, denoted by $\St\left(\Delta\right)$ (resp. $\Speh\left(\Delta\right)$). For example, if $\rho=\mathbbm{1}_1$ is the trivial representation of $G_1$, $x=\frac{1}{2}$ and $y=-\frac{1}{2}$, then $\St\left(\Delta\right)$ is the ordinary Steinberg representation $\St_2$ of $G_2$. As another example, if $\rho=\eta$ is a character of $G_1$, then $\Speh\left(\Delta\right)$ is simply the character 
\[
    \left(\eta|\cdot|^{\frac{x+y}{2}}\right)\circ\det
\]
of $G_{x-y+1}$. The Steinberg representations are essentially discrete series, and they exhaust all essentially discrete series of general linear groups, see \cite[Thm. 9.3]{ZeleII}. 

\vskip 5pt

We often need to compute Jacquet modules and parabolic inductions of segments. Jacquet modules of one segment are given by \cite[Prop. 3.4, Prop. 9.5]{ZeleII}. 

\begin{Prop}\label{P:JacSeg}
Let $\Delta = \left[x,y\right]_\rho$ be a segment, with $\rho$ a supercuspidal representation of $G_d$. Set $m=x-y+1$. For $0\leq \alpha\leq m$, we set $\Delta_I^{(\alpha)}=\left[x,x-\alpha+1\right]_\rho$ and $\Delta_{II}^{(\alpha)}=\left[x-\alpha,y\right]_\rho$. Then we have
\[
    \Jac_{P_k}\St\left(\Delta\right) = \St\left(\Delta_I^{(\alpha)}\right)\boxtimes \St\left(\Delta_{II}^{(\alpha)}\right), \quad
    \Jac_{P_k}\Speh\left(\Delta\right) = \Speh\left(\Delta_{II}^{(m-\alpha)}\right)\boxtimes \Speh\left(\Delta_{I}^{(m-\alpha)}\right) 
\]
if $k=\alpha d$ for some $0\leq \alpha\leq m$, and otherwise these Jacquet modules are zero.
\end{Prop}

\vskip 5pt

We say two segment $\Delta_1=\left[x_1,y_1\right]_{\rho_1}$ and $\Delta_2=\left[x_2,y_2\right]_{\rho_2}$ are \textit{linked}, if $\Delta_1\not\subset \Delta_2$, $\Delta_2\not\subset \Delta_1$, and $\Delta_1\cup\Delta_2$ is still a segment. If $\Delta_1$ and $\Delta_2$ are linked and $x_2\leq x_1$, we shall say that $\Delta_2$ precedes $\Delta_1$. The following results of Zelevinsky \cite[Thm. 4.2, Prop. 4.6]{ZeleII} describe the parabolic induction of segments.

\begin{Thm}\label{T:IndSeg}
Let $\Delta_1, \Delta_2, \cdots, \Delta_r$ be segments.

\begin{enumerate}
    \item The parabolic induction 
    \[
    \St\left(\Delta_1\right)\times\cdots\times\St\left(\Delta_r\right) \quad \left(\textit{resp. }\Speh\left(\Delta_1\right)\times\cdots\times\Speh\left(\Delta_r\right)\right)
    \]
    is irreducible if and only if for each $1\leq i,j\leq r$, $\Delta_i$ and $\Delta_j$ are not linked.

    \vskip 5pt

    \item Suppose that $\Delta_1$ and $\Delta_2$ are linked, and $\Delta_2$ precedes $\Delta_1$. Let $\hat\Delta=\Delta_1\cap\Delta_2$ and $\check\Delta=\Delta_1\cup\Delta_2$. Then $\St\left(\Delta_1\right)\times\St\left(\Delta_2\right)$ is of length two, and there is a short exact sequence
    \[
        0 \lra \St\big(\hat\Delta\big)\times\St\left(\check\Delta\right) \lra \St\left(\Delta_1\right)\times\St\left(\Delta_2\right) \lra \pi \lra 0,
    \]
    where $\pi=LQ\left(\St\left(\Delta_1\right)\times\St\left(\Delta_2\right)\right)$. Likewise, $\Speh\left(\Delta_1\right)\times\Speh\left(\Delta_2\right)$ is also of length two, and we have
    \[
        0 \lra \widehat{\pi} \lra \Speh\left(\Delta_1\right)\times\Speh\left(\Delta_2\right) \lra \Speh\big(\hat\Delta\big)\times\Speh\left(\check\Delta\right) \lra 0,
    \]
    where $\widehat{\pi}$ is the unique subrepresentation of $\Speh\left(\Delta_1\right)\times\Speh\left(\Delta_2\right)$.
\end{enumerate}  
\end{Thm}

\vskip 5pt

Finally we describe general representations. For any representation $\pi$ of $G_n$, by the Langlands classification \cite{MR507262} (also \cite[Thm. 6.1]{ZeleII}), there is a sequence of segments $\Delta_1,\Delta_2,\cdots,\Delta_r$, such that
\[
    \pi\simeq LQ\left(\St\left(\Delta_1\right)\times\St\left(\Delta_2\right)\times\cdots\times\St\left(\Delta_r\right)\right).
\]
Moreover, $\pi$ occurs in the induced representation $\St\left(\Delta_1\right)\times\St\left(\Delta_2\right)\times\cdots\times\St\left(\Delta_r\right)$ with \textit{multiplicity one} by \cite[Thm. 7.1]{ZeleII}.

\vskip 5pt

\subsection{MVW-involution of general linear groups} 
\label{sub:mvw_involution_of_general_linear_groups}

We introduce another useful covariant functor for representations of general linear groups, know as the \textit{MVW-involution}. Recall that there is an outer automorphism of $G_n$ given by
\begin{equation*}
 c: G_n \lra G_n, \quad g \longmapsto w\cdot{}^{t}g^{-1}\cdot w^{-1},
\end{equation*}
where the left superscript ``$t$'' means transpose, and $w$ is a (longest) Weyl group element. This automorphism $c$ induces an \textit{involutive} covariant functor (which we still denote by $c$ by abuse of notations)
\[
    c: \calM\left(G_n\right) \lra \calM\left(G_n\right),
\]
sending a representation $\left(\pi,V\right)$ of $G_n$ to 
\[
    \pi^c: G_n \lra \GL\left(V\right), \quad g\longmapsto \pi\left(c\left(g\right)\right).
\]
Sometimes we shall also write $\pi^{MVW}$ for $\pi^c$. A well-known theorem of Gelfand--Kazhdan says that if $\pi$ is irreducible, then $\pi^c\simeq \pi^\vee$. Moreover, this functor is compatible with parabolic inductions in the sense that
\[
    \left(\pi_1 \times \pi_2\right)^c \simeq \pi_2^c \times \pi_1^c
\]
for any representations $\pi_1$ and $\pi_2$ of general linear groups.

\vskip 5pt

As an application, using this functor and results from Proposition \ref{P:JacSeg} and Theorem \ref{T:IndSeg}, we can compute higher Ext spaces between Steinberg (resp. Speh) representations.

\begin{Lem}\label{L:ExtSteinberg}
Let $\Delta$ and $\Delta'$ be segments, and assume that both $\St\left(\Delta\right)$ and $\St\left(\Delta'\right)$ are representations of $G_n$. Then 
\[
    \Ext_{G_n}^i\left(\St\left(\Delta'\right),\St\left(\Delta\right)\right) = \begin{cases}
        \C \quad & \textit{if $\Delta=\Delta'$ and $i=0,1$};\\[10pt]

        0 \quad & \textit{otherwise}.  
    \end{cases}
\]
Similarly, we have
\[
    \Ext_{G_n}^i\left(\Speh\left(\Delta'\right),\Speh\left(\Delta\right)\right) = \begin{cases}
        \C \quad & \textit{if $\Delta=\Delta'$ and $i=0,1$};\\[10pt]

        0 \quad & \textit{otherwise}.  
    \end{cases}
\]
\end{Lem}

\vskip 5pt

\begin{proof}
We only prove the assertion for Steinberg representations, the assertion for Speh representations is similar. Suppose that $\Delta=\left[x,y\right]_\rho$ and $\Delta'=\left[x',y'\right]_{\rho'}$, where $\rho$ and $\rho'$ are supercuspidal representations of $G_d$ and $G_{d'}$. Trivially, $\Hom_{G_n}\left(\St\left(\Delta'\right),\St\left(\Delta\right)\right)$ is non-zero (hence one dimensional by Schur's lemma) if and only if $\Delta=\Delta'$. So we mainly concern about higher extensions.

\vskip 5pt

Induction on $x-y+1$. The basic case is that when $x=y$, in which case we have $\St\left(\Delta\right) = \rho|\det|^{x}$. By a homological duality theorem \cite[Thm. 2]{Duality} one knows that only when $i=1$ the Ext space can be non-zero, and
\[
    \Ext_{G_n}^1\left(\St\left(\Delta'\right),\rho|\det|^{x}\right) \simeq \Hom_{G_n}\left(\rho|\det|^{x},\St\left(\Delta'\right)\right).
\]
Hence the Ext space above is non-zero and one dimensional if and only if $\Delta=\Delta'$. These proves the basic case.

\vskip 5pt

Assume that the lemma is true for $\Delta_-=\left[x,y+1\right]_\rho$. Note that there is a short exact sequence
\begin{equation}\label{E:ExtSteinberg.SES-1}
    0 \lra \St\left(\Delta\right) \lra \St\left(\Delta_-\right)\times \rho|\det|^y\lra \pi \lra 0,
\end{equation}
where $\pi = LQ\left(\St\left(\Delta_-\right)\times \rho|\det|^y\right)$. By Lemma \ref{L:IndJacAdj}(1) we have
\begin{align*}
 \Ext_{G_n}^i\left(\St\left(\Delta'\right),\St\left(\Delta_-\right)\times \rho|\det|^y\right) & \simeq \Ext_{G_{n-d}\times G_d}^i\left(\Jac_{P_{n-d}}\St\left(\Delta'\right),\St\left(\Delta_-\right)\boxtimes \rho|\det|^y\right).
\end{align*}
According to Proposition \ref{P:JacSeg}, $\Jac_{P_{n-d}}\St\left(\Delta'\right)=0$ unless $d=\alpha d'$ for some $\alpha\in\Z_{>0}$, and in this case
\[
    \Jac_{P_{n-d}}\St\left(\Delta'\right) \simeq \St\left(\Delta'_I\right)\boxtimes \St\left(\Delta'_{II}\right),
\]
where $\Delta'_I$ and $\Delta'_{II}$ are some segment determined by $\Delta'$. Substituting back and applying Theorem \ref{T:Kunneth}, we get
\[
    \Ext_{G_{n-d}\times G_d}^i\left(\Jac_{P_{n-d}}\St\left(\Delta'\right),\St\left(\Delta_-\right)\boxtimes \rho|\det|^y\right) 
\]
\[
    \simeq \bigoplus_{i=j+k}\Ext_{G_{n-d}}^j\left(\St\left(\Delta'_I\right),\St\left(\Delta_-\right)\right)\otimes \Ext_{G_d}^k\left(\St\left(\Delta'_{II}\right),\rho|\det|^{y}\right).
\]
Hence one concludes from the induction hypothesis that $\Ext_{G_n}^i\left(\St\left(\Delta'\right),\St\left(\Delta_-\right)\times \rho|\det|^y\right)$ is non-zero for some $i$ if and only if $\Delta=\Delta'$; in which case
\[
    \Ext_{G_n}^i\left(\St\left(\Delta'\right),\St\left(\Delta_-\right)\times \rho|\det|^y\right) =
    \begin{cases}
       \C \quad & i=0,2;\\[10pt]

       \C^2 \quad & i=1;\\[10pt]

       0 \quad & \textit{otherwise}. 
    \end{cases}
\]
Appyling the functor $\Hom_{G_n}\left(\St\left(\Delta'\right),-\right)$ to the short exact sequence (\ref{E:ExtSteinberg.SES-1}). If $\Delta=\Delta'$, we get a long exact sequence
\begin{flushleft}
$0 \lra \Ext_{G_n}^1\left(\St\left(\Delta'\right),\St\left(\Delta\right)\right) \lra \C^2 \lra \Ext_{G_n}^1\left(\St\left(\Delta'\right),\pi\right)$
\end{flushleft}
\begin{flushright}
$\lra \Ext_{G_n}^2\left(\St\left(\Delta'\right),\St\left(\Delta\right)\right) \lra \C \lra \Ext_{G_n}^2\left(\St\left(\Delta'\right),\pi\right) \lra \Ext_{G_n}^3\left(\St\left(\Delta'\right),\St\left(\Delta\right)\right),$   
\end{flushright}
together with a degree shifting: for all $i\geq 3$,
\[
    \Ext_{G_n}^{i}\left(\St\left(\Delta'\right),\pi\right) \simeq \Ext_{G_n}^{i+1}\left(\St\left(\Delta'\right),\St\left(\Delta\right)\right).
\]
If $\Delta\neq \Delta'$, we get the same degree shifting as above, but for all $i\geq 0$.

\vskip 5pt

Next we apply both the MVW-involution and the contragredient functor to the short exact sequence (\ref{E:ExtSteinberg.SES-1}). We get another short exact sequence
\begin{equation}\label{E:ExtSteinberg.SES-2}
    0 \lra \pi \lra  \rho|\det|^y\times\St\left(\Delta_-\right)\lra \St\left(\Delta\right) \lra 0.
\end{equation}
Similar to the argument above, by Lemma \ref{L:IndJacAdj}, Proposition \ref{P:JacSeg} and the induction hypothesis, one can conclude that
\[
    \Ext_{G_n}^i\left(\St\left(\Delta'\right),\rho|\det|^y\times\St\left(\Delta_-\right)\right) = 0
\]
for all $i\geq 0$. Again by appyling the functor $\Hom_{G_n}\left(\St\left(\Delta'\right),-\right)$ to the short exact sequence (\ref{E:ExtSteinberg.SES-2}), we get
\[
    \Ext_{G_n}^i\left(\St\left(\Delta'\right),\St\left(\Delta\right)\right) \simeq \Ext_{G_n}^{i+1}\left(\St\left(\Delta'\right),\pi\right)
\]
for all $i\geq 0$. The desire conclusion then follows from the combination of this degree shifting and the long exact sequence/ degree shifting above. These complete the proof.

\end{proof}

\vskip 5pt

\begin{Rmk}
In the case that $\rho\simeq\rho'$ is a character of $G_1$, this lemma is a consequence of \cite[Cor. 2]{MR2173717}.
\end{Rmk}

\vskip 5pt

\subsection{Local zeta integrals of Godement--Jacquet} 
\label{sub:local_zeta_integrals_of_godement_jacquet}

Finally we briefly recall the theory of local zeta integrals developed by Godement--Jacquet \cite{MR342495}. Let $\left(\pi, V\right)$ be an irreducible representation of $G_n$. For a complex variable $s\in\C$, a test function $\varphi\in\calS\left(M_{n,n}\right)$ and a matrix coefficient $f$ of $\pi$, one can define the zeta integral
\[
    Z\left(s,\varphi,f\right) = \int_{G_n}\varphi\left(g\right) f\left(g\right)|\det g|^{s+\frac{n-1}{2}}\, dg,
\]
where $dg$ is a fixed Haar measure of $G_n$. This integral is absolutely convergent when ${\rm Re}\left(s\right)$ is sufficiently large, and admits a meromorphic continuation to the whole complex plane $\C$. Indeed, one knows that $Z\left(s,\varphi,f\right)$ is a rational function in $\C\left(q^{-s}\right)$. Let $\varphi$ and $f$ vary. Then all these zeta integrals form a finitely generated fractional ideal of $\C\left[q^s,q^{-s}\right]$. Godement and Jacquet showed that there exists a \textit{greatest common divisor} of all these zeta integrals, which is of the form
\[
    L\left(s,\pi\right) = \frac{1}{P\left(q^{-s}\right)},
\]
with $P\left(X\right)$ a polynomial in $X$ independent of $\varphi$ or $f$, and the constant term of $P\left(X\right)$ equals to $1$. This $L\left(s,\pi\right)$ is the so called (local) Godement--Jacquet L-function associated to $\pi$. We need the following results  \cite[Prop. 3.5, Cor. 3.6, Prop. 5.11]{MR342495} for later proofs.

\begin{Prop}\label{P:GJLfunction}
Suppose that $\pi = LQ\left(\St\left(\Delta_1\right)\times\St\left(\Delta_2\right)\times\cdots\times\St\left(\Delta_r\right)\right)$, with $\Delta_i=\left[x_i,y_i\right]_{\rho_i}$. 

\begin{enumerate}
    \item We have
\[
    L\left(s,\pi\right) = \prod_{i=1}^{r} L\left(s+x_i,\rho_i\right),
\]
where
\[
    L\left(s,\rho_i\right) = \begin{cases}
       \frac{1}{1-\rho_i\left(\varpi\right)q^{-s}} \quad & \textit{if $\rho_i$ is a unramified character of $G_1$};\\[10pt]

       1 \quad & \textit{otherwise}. 
    \end{cases}
\]

\vskip 5pt

\item For any irreducible subquotient $\pi'$ of $\St\left(\Delta_1\right)\times\St\left(\Delta_2\right)\times\cdots\times\St\left(\Delta_r\right)$, the fraction $L\left(s,\pi'\right)/L\left(s,\pi\right)$ is an entire function.
\end{enumerate}
\end{Prop}

\vskip 10pt

\section{The Weil representation} 
\label{sec:the_weil_representation}

Recall that we have defined the Weil representation $\omega_{n,m}$ of $G_n\times H_m$ in (\ref{E:DefWeilRep}). In this section we take a closer look at this representation, and collect some useful results to study this representation.

\vskip 5pt

\subsection{MVW-involution as Fourier transformation} 
\label{sub:mvw_involution_as_fourier_transformation}

In Section \ref{sub:mvw_involution_of_general_linear_groups} we have introduced the functor of MVW-involution. This functor can be extended to the category of representations of $G_n\times H_m$. Consider the outer automorphism
\[
    c: G_n\times H_m \lra G_n\times H_m, \quad \left(g,h\right) \longmapsto \left(w_n{}^{t}g^{-1} w_n^{-1}, w_m{}^{t}h^{-1} w_m^{-1}\right).
\]
Here as before, the superscript ``$t$'' means transpose, and $w_n$ (resp. $w_m$) is a (longest) Weyl group element of $G_n$ (resp. $H_m$). This outer automorphism induces a functor
\[
    c: \calM\left(G_n\times H_m\right) \lra \calM\left(G_n\times H_m\right), \quad \pi \longmapsto \pi^c.
\]
Sometimes we would like to write $\pi^{MVW}$ for $\pi^c$. The point here is that for the Weil representation $\omega_{n,m}$, the MVW-involution can be alternatively realized as Fourier transformation.

\begin{Lem}\label{L:MVW4rier}
As representations of $G_n\times H_m$, we have
\[
    \omega_{n,m}^{MVW} \simeq \omega_{n,m}.
\]   
\end{Lem}

\vskip 5pt

\begin{proof}
This is \cite[Lem. 4.3]{MR3930015}; the isomorphism is given by the Fourier transformation
\begin{align*}
    \calS\left(M_{n,m}\right) &\lra \calS\left(M_{n,m}\right),
\end{align*}  
sending a test function $\varphi\in\calS\left(M_{n,m}\right)$ to 
\[
    X\longmapsto \int_{M_{n,m}}\varphi\left(Y\right)\psi_F\left({\rm tr}_m\left({}^tX\cdot Y\right)\right)dY.
\]
Here $\psi_F$ is a non-trivial additive character of $F$, and ${\rm tr}_m$ is the trace map of $M_{m,m}$.

\end{proof}

\vskip 5pt

Consequently, we deduce that:

\begin{Cor}\label{C:ExtvsMVW}
For any representation $\pi$ of $G_n$, we have
\[
    \Ext_{G_n}^i\left(\omega_{n,m},\pi^c\right)_{sm} \simeq \Ext_{G_n}^i\left(\omega_{n,m},\pi\right)^c_{sm}
\] 
as representations of $H_m$, for all $i\geq 0$.  
\end{Cor}

\vskip 5pt

\begin{proof}
Using the previous lemma, one can check by hand that
\[
    \Hom_{G_n}\left(\omega_{n,m},-\right)_{sm}\circ c \simeq c \circ \Hom_{G_n}\left(\omega_{n,m},-\right)_{sm}
\]
as functors from $\calM\left(G_n\right)$ to $\calM\left(H_m\right)$. The desired conclusion then follows from deriving both sides of above. 

\end{proof}

\vskip 5pt


\subsection{The rank filtration on the Weil representation} 
\label{sub:the_rank_filtration}

Since the Weil representation is realized on the space of Schwartz functions on the matrices, there is a natural filtration induced by ranks of matrices. To be more precise, we have a $G_n\times H_m$-invariant filtration on $\calS\left(M_{n,m}\right)$:

\[
    0=R_{a+1} \subset R_a \subset R_{a-1} \subset \cdots \subset R_0 = \calS\left(M_{n,m}\right).
\]
Here $a=\min\left\{n,m\right\}$, and for each $0\leq k\leq a$, $R_k$ is the space of Schwartz functions supported on the matrices in $M_{n,m}$ of rank greater than or equal to $k$. Let $\Omega_k$ be the set of matrices in $M_{n,m}$ of rank $k$. Then as vector spaces, we have
\[
    R_k/R_{k+1} \simeq \calS\left(\Omega_k\right).
\]
We shall denote by $\tau_k$ the action of $G_n\times H_m$ on $R_k/R_{k+1}$ inherited from $\omega_{n,m}$. According to \cite[Sect. 2]{MR2504432} (or \cite[Sect. 5.1]{MR3753906}), this representation $\tau_k$ can be described as follows.

\begin{Lem}\label{L:rank-filtration}
Let $P_k$ be the standard parabolic subgroup of $G_n$ with Levi component $G_k\times G_{n-k}$, and $Q_k$ the standard parabolic subgroup of $H_m$ with Levi component $H_k\times H_{m-k}$. Denote by $\overline{Q}_k$ the parabolic subgroup of $H_m$ opposite to $Q_k$. Then we have
\[
    \tau_k \simeq \Ind_{P_k\times \overline{Q}_k}^{G_n\times H_m} \left(\omega^\natural_k\otimes\xi_k\right),
\]
where the parabolic induction is normalized, and:
\begin{itemize}
     \item $\omega^\natural_k$ is the natural geometric action of $G_k\times H_k$ on $\calS\left(G_k\right)$, namely
     \[
        \left(\omgn_k\left(g,h\right)\varphi\right) \left(X\right) = \varphi\left(g^{-1}Xh\right)
     \]
     for $\left(g,h\right)\in G_k\times H_k$, $\varphi\in\calS\left(G_k\right)$ and $X\in G_k$;

     \vskip 5pt

     \item $\xi_k$ is a character of the standard Levi subgroup of $P_k\times \overline{Q}_k$, given by
     \[  \xi_k = 
     \begin{cases}
       |\det|^{\frac{k-n-m}{2}}  \quad & \textit{on } G_{k} ; \\[10pt]

       |\det|^{\frac{k-m}{2}}  \quad & \textit{on } G_{n-k} ; \\[10pt]

       |\det|^{\frac{n+m-k}{2}}  \quad & \textit{on } H_{k} ; \\[10pt]

       |\det|^{\frac{n-k}{2}}  \quad & \textit{on } H_{m-k}.
     \end{cases}
     \]
\end{itemize} 
\end{Lem}

\vskip 5pt

Using this filtration, Adams--Prasad--Savin showed that for any irreducible representation $\pi$ of $G_n$ and any $*\geq 0$, $\Ext_{G_n}^*\left(\omega_{n,m},\pi\right)_{sm}$ is a finite length representation of $H_m$, hence $\EP_{G_n}\left(\omega_{n,m},\pi\right)_{sm}$ is well-defined in the Grothendieck group. We refer readers to \cite[Prop. 1.1, Prop. 5.17]{MR3753906} for more details.

\vskip 5pt


\subsection{Jacquet modules of the Weil representation} 
\label{sub:kudla_s_filtration}

In our later computations, we often need to deal with Jacquet modules of the Weil representation. The \textit{Kudla's filtration} \cite[Prop. 3.2]{MR2504432} provides a fruitful description of these Jacquet modules. For each $0\leq k\leq n$, again we let $P_k$ be the standard parabolic subgroup of $G_n$ with Levi component $M_k = G_k\times G_{n-k}$.

\begin{Prop}\label{P:KudlaFiltration}
The normalized Jacquet module $\Jac_{P_k}\omega_{n,m}$ of $\omega_{n,m}$ has an $M_k\times H_m$-invariant filtration
\[
    0= J_{k+1}\subset J_k \subset J_{k-1} \subset \cdots \subset J_0 = \Jac_{P_k}\omega_{n,m},
\]
and for each $0\leq i \leq k$, the successive quotient $\lambda_i=J_i/J_{i+1}$ is given by
\[
    \lambda_i \simeq \Ind_{P_{k-i,i}\times G_{n-k} \times Q_i}^{M_k\times H_m}\left(\mu_{k,i}\otimes \omgn_i\otimes\omega_{n-k,m-i}\right).
\]
Here:
\begin{itemize}
    \item $P_{k-i,i}$ is the standard parabolic subgroup of $G_k$ with Levi component $G_{k-i}\times G_i$;

    \vskip 5pt

    \item $Q_i$ is the standard parabolic subgroup of $H_m$ with Levi component $H_{i}\times H_{m-i}$; if $Q_i$ does not exists, then $\lambda_i$ is interpreted to be zero;

    \vskip 5pt

    \item $\omgn_i$ is the natural geometric action of $G_i\times H_i$ on $\calS\left(G_i\right)$;

    \vskip 5pt

    \item $\omega_{n-k,m-i}$ is the Weil representation associated to $G_{n-k}\times H_{m-i}$;

    \vskip 5pt

    \item $\mu_{k,i}$ is a character of the standard Levi subgroup of $P_{k-i,i}\times G_{n-k} \times Q_i$, given by
    \[  \mu_{k,i} = 
     \begin{cases}
       |\det|^{\frac{m-n+k-i}{2}}  \quad & \textit{on } G_{k-i} ; \\[10pt]

       |\det|^{\frac{m-n+2k-i}{2}}  \quad & \textit{on } G_{i} ; \\[10pt]

       |\det|^{\frac{k-i}{2}}  \quad & \textit{on } G_{n-k} ; \\[10pt]

       |\det|^{\frac{n-m-2k+i}{2}}  \quad & \textit{on } H_{i} ; \\[10pt]

       |\det|^{\frac{i-k}{2}}  \quad & \textit{on } H_{m-i}.
     \end{cases}
     \]
\end{itemize}
\end{Prop}

\vskip 5pt

We end up this section with a standard application of the Kudla's filtration. Let $\Delta = \left[x,y\right]_\rho$ be a segment. Assume that $\St\left(\Delta\right)$ is an irreducible representation of $G_k$.

\begin{Cor}\label{C:KudlaIndCompatible}
Suppose that either $\rho\not\simeq\mathbbm{1}_1$, or $y\neq \frac{1+m-n}{2}$. Then for any representation $\pi_0$ of $G_{n-k}$, and any $i\geq 0$, we have
\[
    \Ext_{G_n}^i\left(\omega_{n,m}, \St\left(\Delta\right)\times\pi_0\right)_{sm} \simeq \St\left(\Delta\right)\times \Ext_{G_{n-k}}^i\left(\omega_{n-k,m-k}, \pi_0\right)_{sm} 
\]
as representations of $H_m$.
\end{Cor}

\vskip 5pt

\begin{proof}
The proof is similar to that of \cite[Prop. 5.2]{MR3714507}. By the Frobenius reciprocity (see Lemma \ref{L:IndJacAdj}(1)), we have
\[
    \Ext_{G_n}^i\left(\omega_{n,m}, \St\left(\Delta\right)\times\pi_0\right)_{sm} \simeq \Ext_{G_k\times G_{n-k}}^i\left(\Jac_{P_k}\omega_{n,m}, \St\left(\Delta\right)\boxtimes\pi_0\right)_{sm}.
\]
To compute the right hand side, we appeal to the Kudla's filtration. For all $0\leq j\leq k-1$ and $*\geq 0$, we have
\[
    \Ext_{G_k\times G_{n-k}}^*\left(\lambda_j, \St\left(\Delta\right)\boxtimes\pi_0\right)_{sm}
\]
\begin{align*}
     &\simeq \Ind_{Q_j}^{H_m}\left(\Ext_{G_k\times G_{n-k}}^*\left(\Ind_{P_{k-j,j}\times G_{n-k}}^{G_k\times G_{n-k}}\left(\mu_{k,j}\otimes\omgn_j\otimes\omega_{n-k,m-j}\right), \St\left(\Delta\right) \boxtimes \pi_0\right)\right).
\end{align*} 
Then by the second adjointness (see Lemma \ref{L:IndJacAdj}(2)), the latter is isomorphic to
\begin{equation}\label{E:IndExt}
    \Ind_{Q_j}^{H_m}\left(\Ext_{G_{k-j}\times G_j\times G_{n-k}}^*\left(\mu_{k,j}\otimes\omgn_j\otimes\omega_{n-k,m-j}, \Jac_{\overline{P}_{k-j,j}}\left(\St\left(\Delta\right)\right)\boxtimes \pi_0\right)\right).
\end{equation}
Now, suppose that 
\[
    \Jac_{\overline{P}_{k-j,j}}\left(\St\left(\Delta\right)\right) = \delta_1\boxtimes\delta_2,
\]
where $\delta_1$ and $\delta_2$ are essentially discrete series of $G_{k-j}$ and $G_j$, respectively. If $\rho\not\simeq\mathbbm{1}_1$, then by Lemma \ref{L:ExtSteinberg} and \cite[Thm. 2]{Duality}, as representations of $G_{k-j}$, there is no Hom or Ext between $\delta_1$ and $\mu_{k,j}$. On the other hand, if $\rho\simeq\mathbbm{1}_1$ but $y\neq \frac{1+m-n}{2}$, by Proposition \ref{P:JacSeg} we have
\[
    \delta_1 = \St\left(\Delta_{II}^{(j)}\right),
\]
where $\Delta_{II}^{(j)} = \left[y+k-j-1,y\right]_{\mathbbm{1}_1}$.
In this case easy to see that the central character of $\delta_1$ is different from that of $\mu_{k,j}$. Therefore under our assumption, the Ext space in (\ref{E:IndExt}) is always zero. This implies that
\begin{align*}
\Ext_{G_k\times G_{n-k}}^i\left(\Jac_{P_k}\omega_{n,m}, \St\left(\Delta\right)\boxtimes\pi_0\right)_{sm} &\simeq \Ext_{G_k\times G_{n-k}}^i\left(\lambda_k, \St\left(\Delta\right)\boxtimes\pi_0\right)_{sm}\\
&\simeq \St\left(\Delta\right) \times \Ext_{G_{n-k}}^i\left(\omega_{n-k,m-k}, \pi_0\right)_{sm}
\end{align*}
as desired. This completes the proof.

\end{proof}

\vskip 10pt



\section{Proof of the main theorem I} 
\label{sec:proof_of_the_main_theorem_i}

This section is devoted to the proof of the main result I: Theorem \ref{T:Main-1}. So in this section, we assume that $n\leq m$. In particular, we have $\frac{1+m-n}{2}>0$.

\vskip 5pt

\subsection{Holomorphicity of the L-function implies Ext-vanishing} 
\label{sub:holomorphicity_of_the_l_function_implies_ext_vanishing}

We start with an easy computation of the local L-function. Let
\begin{equation}\label{E:ProofMain-IpiLQ}
    \pi = LQ\left(\St\left(\Delta_1\right)\times\St\left(\Delta_2\right)\times\cdots\times\St\left(\Delta_r\right)\right)
\end{equation} 
be an irreducible representation of $G_n$, with $\Delta_i=\left[x_i,y_i\right]_{\rho_i}$ for $1\leq i \leq r$. Let $\check\Delta_i=\left[-y_i,-x_i\right]_{\rho_i^\vee}$. Then we have
\[
    \pi^\vee = LQ\left(\St\left(\check\Delta_r\right)\times \cdots\times \St\left(\check\Delta_2\right)\times \St\left(\check\Delta_1\right)\right).
\]
By Proposition \ref{P:GJLfunction}(1), the L-function associated to $\pi^\vee$ is 
\[
    L\left(s,\pi^\vee\right) = \prod_{i=1}^{r} L\left(s-y_i,\rho_i^\vee\right).
\]
It follows that $L\left(s,\pi^\vee\right)$ is holomorphic at $s=\frac{1+m-n}{2}$ if and only if for any $1\leq i \leq r$, either $\rho_i\not\simeq\mathbbm{1}_1$, or $y_i\neq \frac{1+m-n}{2}$. Note that this condition is exactly the one appears in Corollary \ref{C:KudlaIndCompatible}! Therefore for all $*>0$ we have
\[
    \Ext_{G_n}^*\left(\omega_{n,m},\St\left(\Delta_1\right)\times\St\left(\Delta_2\right)\times\cdots\times\St\left(\Delta_r\right)\right)_{sm}
\]
\[
    \simeq \St\left(\Delta_1\right)\times\St\left(\Delta_2\right)\times\cdots\times\St\left(\Delta_r\right) \times \Ext_{G_0}^*\left(\omega_{0,m-n},\mathbbm{1}_0\right)_{sm} = 0,
\]
since the trivial group $G_0$ is compact. Based on this computation, we deduce that:

\begin{Prop}\label{P:Holo-L-vanish-Ext}
Suppose that $n\leq m$, and either $L\left(s,\pi\right)$ or $L\left(s,\pi^\vee\right)$ is holomorphic at $s=\frac{1+m-n}{2}$. Then
\[
    \Ext_{G_n}^*\left(\omega_{n,m},\pi\right)_{sm}=0
\] 
for all $*> 0$.  

\end{Prop}

\vskip 5pt

\begin{proof}
Let $\pi$ be of the form as in (\ref{E:ProofMain-IpiLQ}), and we fix a degree $*>0$. We first assume that $L\left(s,\pi^\vee\right)$ is holomorphic at $s=\frac{1+m-n}{2}$. Consider the short exact sequence
\[
    0 \lra \varkappa \lra \Pi \xlongrightarrow[\quad]{p} \pi \lra 0,
\]
where $\Pi=\St\left(\Delta_1\right)\times\St\left(\Delta_2\right)\times\cdots\times\St\left(\Delta_r\right)$, the map $p$ is the natural projection, and $\varkappa$ is the kernel of $p$. Applying the functor $\Hom_{G_n}\left(\omega_{n,m},-\right)_{sm}$ to it, the associated long exact sequence reads
\[
    \cdots\lra\Ext_{G_n}^*\left(\omega_{n,m},\Pi\right)_{sm} \lra \Ext_{G_n}^*\left(\omega_{n,m},\pi\right)_{sm} \lra \Ext_{G_n}^{*+1}\left(\omega_{n,m},\varkappa\right)_{sm}\lra\cdots.
\]
As in the discussion before this proposition, we have shown that the holomorphicity of $L\left(s,\pi^\vee\right)$ implies that
\[
    \Ext_{G_n}^*\left(\omega_{n,m},\Pi\right)_{sm} = 0.
\]
Therefore in order to prove the desired conclusion it suffices to show that for any irreducible subquotient $\pi'$ of $\varkappa$, we have 
\begin{equation*}
    \Ext_{G_n}^{*+1}\left(\omega_{n,m},\pi'\right)_{sm} = 0.
\end{equation*}
Now note that by Proposition \ref{P:GJLfunction}(2), the L-function $L\left(s,\left(\pi'\right)^\vee\right)$ is still holomorphic at the point $s=\frac{1+m-n}{2}$. So we can repeat this procedure, and eventually we can reduce the desired conclusion to showing that
\begin{equation}\label{E:Main-IExtvanishshift1}
    \Ext_{G_n}^{*+n}\left(\omega_{n,m},\pi''\right)_{sm} = 0
\end{equation}
for all $\pi''$ in a certain finite set of irreducible representations. By \cite[Pg. 98, Sect. 4.2]{BNote}, we do know that (\ref{E:Main-IExtvanishshift1}) holds, as the degree $*+n$ is greater than the split rank of $G_n$.

\vskip 5pt

Next we assume that $L\left(s,\pi\right)$ is holomorphic at $s=\frac{1+m-n}{2}$ but $L\left(s,\pi^\vee\right)$ is not. 
Consider the MVW-involution of $\pi$. It follows from Corollary \ref{C:ExtvsMVW} that
\[
    \Ext_{G_n}^*\left(\omega_{n,m},\pi\right)_{sm} \simeq \left(\Ext_{G_n}^*\left(\omega_{n,m},\pi^\vee\right)_{sm} \right)^c .
\]
Then by the discussion above, we know that 
\[
    \Ext_{G_n}^*\left(\omega_{n,m},\pi^\vee\right)_{sm} = 0
\]
for all $*>0$. Therefore in this case the desired conclusion still holds.

\end{proof}

\vskip 5pt


\subsection{Zeta integrals in non-equal rank cases} 
\label{sub:zeta_integrals_in_the_non_equal_rank_case}

Now we modify the zeta integral to non equal rank cases. By using the decomposition $M_{n,m}\simeq M_{n,n}\times M_{n,m-n}$, we can write any element $\tilde{X}\in M_{n,m}$ as a block matrix
\[
    \big(X_0 ~|~ X_1\big),
\]
where $X_0\in M_{n,n}$, and $X_1\in M_{n,m-n}$. Let $\varphi\in \calS\left(M_{n,m}\right)$, we set $\varphi^\flat\in\calS\left(M_{n,n}\right)$ to be the function
\begin{equation}\label{E:Schwartzfunction-flat}
    \varphi^\flat\left(X\right) = \varphi\left(X~|~0\right)
\end{equation}
for $X\in M_{n,n}$. Then for a given $\varphi \in \calS\left(M_{n,m}\right)$ and a matrix coefficient $f$ of $\pi$, we can consider the integral
\begin{align*}
Z^\flat\left(s,\varphi,f\right)  \coloneqq{}&  Z\left(s,\varphi^\flat, f\right) \\
 ={}& \int_{G_n} \varphi\left(X~|~0\right) f\left(g\right) |\det g|^{s+\frac{n-1}{2}}\,dg.
\end{align*}
By \cite{MR342495}, we know that this integral is absolutely convergent when ${\rm Re}\left(s\right)$ is sufficiently large, and admits a meromorphic continuation to the whole plane. Moreover, the fraction
\[
    Z^\natural\left(s,\varphi,f\right) \coloneqq \frac{Z^\flat\left(s,\varphi,f\right)}{L\left(s,\pi\right)}
\]
is an entire function. Therefore we get a well-defined map
\begin{equation}\label{E:non-equal-rank-Zeta}
Z^\natural\left(s,-,-\right): \omega_{n,m} \lra \pi\boxtimes\pi^\vee.
\end{equation}

\begin{Lem}\label{L:construction-non-equal-rank-Zeta}
This map (\ref{E:non-equal-rank-Zeta}) is $G_n\times \overline{Q}_{n}$-equivariant and non-zero. Here: 

\begin{itemize}
    \item $Q_n$ is the standard parabolic subgroup of $H_m$ whose Levi component is $H_n \times H_{m-n}$, and $\overline{Q}_n$ is the parabolic subgroup of $H_m$ opposite to $Q_n$;

    \vskip 5pt

    \item $G_n\times \overline{Q}_{n}$ acts on the left hand side of this map by $\omega_{n,m}$;

    \vskip 5pt

    \item $G_n$ acts on the right hand side of this map by $\pi|\det|^{s-\frac{1+m-n}{2}}$;

    \vskip 5pt

    \item the action of $\overline{Q}_n$ on the right hand side factor through the Levi component $H_n \times H_{m-n}$, the factor $H_n$ acts by $\pi^\vee|\det|^{-s+\frac{1}{2}}$, and $H_{m-n}$ acts by the character $|\det|^{\frac{n}{2}}$.
\end{itemize}
\vskip 5pt
In particular, by the Frobenius reciprocity, this map (\ref{E:non-equal-rank-Zeta}) gives rise to a non-zero element in the Hom space
\[
    \Hom_{G_n\times H_m}\left(\omega_{n,m}, \pi|\det|^{s-\frac{1+m-n}{2}}\boxtimes \Ind_{\overline{Q}_n}^{H_m}\left(\pi^\vee|\det|^{-s+\frac{1+m-n}{2}}\boxtimes\mathbbm{1}_{m-n}\right)\right).
\]
  
\end{Lem}

\vskip 5pt

\begin{proof}
The equivariant property follows from the definition of $Z^\flat\left(s,-,-\right)$ directly. To see that the map is non-zero, one just needs to note that $\varphi \mapsto \varphi^\flat$ is surjective.

\end{proof}

\vskip 5pt

Specializing this construction to $s=\frac{1+m-n}{2}$, we get a non-zero element $Z^\natural\left(\frac{1+m-n}{2},-,-\right)$ in the Hom space
\[
    \Hom_{G_n\times H_m}\left(\omega_{n,m}, \pi\boxtimes \Ind_{\overline{Q}_n}^{H_m}\left(\pi^\vee\boxtimes\mathbbm{1}_{m-n}\right)\right).
\]
Note that if $\pi = LQ\left(\St\left(\Delta_1\right)\times\St\left(\Delta_2\right)\times\cdots\times\St\left(\Delta_r\right)\right)$, with $\Delta_i=\left[x_i,y_i\right]_{\rho_i}$ and they are properly arranged according to their exponents, then $\Ind_{\overline{Q}_n}^{H_m}\left(\pi^\vee\boxtimes\mathbbm{1}_{m-n}\right) \simeq \mathbbm{1}_{m-n} \times \pi^\vee$ is indeed a subrepresentation of the parabolic induction
\[
    |\cdot|^{-\frac{m-n-1}{2}}\times\cdots\times |\cdot|^{\frac{m-n-1}{2}}\times \St\left(\Delta_1\right)^\vee\times\St\left(\Delta_2\right)^\vee\times\cdots\times\St\left(\Delta_r\right)^\vee
\]
occurring in Theorem \ref{T:HoweDuality}(2).
By the Howe duality and the multiplicity one property of the Langlands quotient, we know that there exists two subrepresentations $\Sigma_0\subset\Sigma_1$ of $\Ind_{\overline{Q}_n}^{H_m}\left(\pi^\vee\boxtimes\mathbbm{1}_{m-n}\right)$, such that:
\begin{itemize}
    \item the quotient $\Sigma_1/\Sigma_0$ is isomorphic to $\theta_{n,m}\left(\pi\right)$;

    \vskip 5pt

    \item any subquotient of $\Ind_{\overline{Q}_n}^{H_m}\left(\pi^\vee\boxtimes\mathbbm{1}_{m-n}\right)/\Sigma_1$ and $\Sigma_0$ is not isomorphic to $\theta_{n,m}\left(\pi\right)$.
\end{itemize}
\vskip 5pt
A priori we do not know that whether $\Sigma_0$ and $\Sigma_1$ are unique. Nevertheless we have the following simple observation.

\begin{Prop}
The map $Z^\natural\left(\frac{1+m-n}{2},-,-\right)$ indeed gives rise to a non-zero element in the Hom space
\[
    \Hom_{G_n\times H_m}\left(\omega_{n,m}, \pi\boxtimes \theta_{n,m}\left(\pi\right)\right),
\]  
in the sense that:

\begin{itemize}
    \item $Z^\natural\left(\frac{1+m-n}{2},-,-\right)$ lies in the subspace
    \[
        \Hom_{G_n\times H_m}\left(\omega_{n,m}, \pi\boxtimes \Sigma_1\right) \hookrightarrow \Hom_{G_n\times H_m}\left(\omega_{n,m}, \pi\boxtimes \left(\mathbbm{1}_{m-n} \times \pi^\vee\right) \right);
    \]

    \vskip 5pt

    \item the image of $Z^\natural\left(\frac{1+m-n}{2},-,-\right)$ under the natural map
    \[
        \Hom_{G_n\times H_m}\left(\omega_{n,m}, \pi\boxtimes \Sigma_1\right) \lra \Hom_{G_n\times H_m}\left(\omega_{n,m}, \pi\boxtimes \Sigma_1/\Sigma_0\right) 
    \]
    is non-zero.
\end{itemize}
\end{Prop}

\vskip 5pt

\begin{proof}
The proposition simply follows from Theorem \ref{T:HoweDuality}(1).

\end{proof}

\vskip 5pt

With this observation at hand, finally we can mimic the argument of Fang--Sun--Xue \cite[Thm. 1.3]{FSX} to complete the proof of Theorem \ref{T:Main-1}. Let $\Omega_n$ be the set of matrices in $M_{n,m}$ of rank $n$, which is stable under the action of $G_n\times H_m$. We denote by $\omega^0_{n,m}$ the subrepresentation of $\omega_{n,m}$ with the underlying vector space $\calS\left(\Omega_n\right)$. 

\begin{Prop}\label{P:Holo-L-vanish-boundary}
Suppose that the L-function $L\left(s,\pi\right)$ is holomorphic at $s=\frac{1+m-n}{2}$. Then we have 
\[
    \Hom_{G_n}\left(\omega_{n,m}/\omega^0_{n,m},\pi\right) = 0.
\]    
\end{Prop}

\vskip 5pt

\begin{proof}
We first claim that: under the holomorphicity assumption of the L-function $L\left(s,\pi\right)$, we have $\Sigma_1 = \Ind_{\overline{Q}_n}^{H_m}\left(\pi^\vee\boxtimes\mathbbm{1}_{m-n}\right)$,
and $\Sigma_0$ is the unique maximal subrepresentation of $\Sigma_1$. Moreover, the natural map
\begin{equation*}
    \mathscr{R}_{\theta}: \Hom_{G_n\times H_m}\left(\omega_{n,m},\pi\boxtimes\theta_{n,m}\left(\pi\right)\right) \lra \Hom_{G_n\times H_m}\left(\omega_{n,m}^0,\pi\boxtimes\theta_{n,m}\left(\pi\right)\right)
\end{equation*}   
induced by the inclusion $\omega^0_{n,m}\hookrightarrow\omega_{n,m}$ is non-zero.

\vskip 5pt

To proof our claim, consider the restriction of $Z^\natural\left(\frac{1+m-n}{2},-,-\right)$ to $\omega^0_{n,m}$, namely, the image of $Z^\natural\left(\frac{1+m-n}{2},-,-\right)$ under the map
\begin{equation*}
 \mathscr{R}_{ind}: \Hom_{G_n\times H_m}\left(\omega_{n,m}, \pi\boxtimes \Ind_{\overline{Q}_n}^{H_m}\left(\pi^\vee\boxtimes\mathbbm{1}\right)\right) \lra \Hom_{G_n\times H_m}\left(\omega^0_{n,m}, \pi\boxtimes \Ind_{\overline{Q}_n}^{H_m}\left(\pi^\vee\boxtimes\mathbbm{1}\right)\right).   
\end{equation*}
For any $\varphi\in\calS\left(\Omega_n\right)$, we have $\varphi^\flat\in\calS\left(G_n\right)$, and $\varphi\mapsto\varphi^\flat$ induces a surjection $\calS\left(\Omega_n\right) \twoheadrightarrow \calS\left(G_n\right)$.
Hence for any matrix coefficient $f$ of $\pi$, the integral $Z^\flat\left(s,\varphi,f\right) = Z\left(s,\varphi^\flat,f\right)$ is absolutely convergent and holomorphic for all $s\in\C$.
Since the L-function $L\left(s,\pi\right)$ is holomorphic at $s=\frac{1+m-n}{2}$, the restriction of $Z^\natural\left(\frac{1+m-n}{2},-,-\right)$ to $\omega^0_{n,m}$ is non-zero. This implies that the map $\mathscr{R}_{ind}$ is non-zero.

\vskip 5pt

Now note that the target of $\mathscr{R}_{ind}$ can be rewritten as
\begin{align*}
\Hom_{G_n\times H_m}\left(\omega^0_{n,m}, \pi\boxtimes \Ind_{\overline{Q}_n}^{H_m}\left(\pi^\vee\boxtimes\mathbbm{1}\right)\right) & \simeq \Hom_{H_m}\left(\left(\omega^0_{n,m}\otimes\pi^\vee\right)_{G_n} , \Ind_{\overline{Q}_n}^{H_m}\left(\pi^\vee\boxtimes\mathbbm{1}\right)\right) \\
& \simeq \Hom_{H_m}\left(\Ind_{\overline{Q}_n}^{H_m}\left(\pi^\vee\boxtimes\mathbbm{1}\right) , \Ind_{\overline{Q}_n}^{H_m}\left(\pi^\vee\boxtimes\mathbbm{1}\right)\right).
\end{align*}
It is known that this parabolic induction $\Ind_{\overline{Q}_n}^{H_m}\left(\pi^\vee\boxtimes\mathbbm{1}_{m-n}\right) \simeq \mathbbm{1}_{m-n} \times \pi^\vee$ is \textit{socle irreducible} (namely, its socle is irreducible and occurs with multiplicity one, see \cite[Def. 2.3, Def. 4.3, Cor. 4.9]{MR3573961}), thus any non-zero element in the Hom space
\[
    \Hom_{H_m}\left(\Ind_{\overline{Q}_n}^{H_m}\left(\pi^\vee\boxtimes\mathbbm{1}_{m-n}\right) , \Ind_{\overline{Q}_n}^{H_m}\left(\pi^\vee\boxtimes\mathbbm{1}_{m-n}\right)\right)
\]
is indeed an isomorphism. Therefore, if we re-interpret $Z^\natural\left(\frac{1+m-n}{2},-,-\right)$ as a map
\[
    Z^\natural_{n,m}:\omega_{n,m}\otimes\pi^\vee \lra \Ind_{\overline{Q}_n}^{H_m}\left(\pi^\vee\boxtimes\mathbbm{1}_{m-n}\right),
\]
then this map is surjective, as its restriction to $\omega^0_{n,m}\otimes\pi^\vee$ is already surjective. The surjectivity of $Z^\natural_{n,m}$ implies that $\Sigma_1 = \Ind_{\overline{Q}_n}^{H_m}\left(\pi^\vee\boxtimes\mathbbm{1}_{m-n}\right)$, and the Howe duality then implies that $\Sigma_0$ is the unique maximal subrepresentation of $\Sigma_1$. Moreover, it follows from the surjectivity of $Z^\natural_{n,m}\,\big|_{\omega^0_{n,m}\otimes\pi^\vee}$ that the map $\mathscr{R}_{\theta}$ is non-zero. This completes the proof of our claims.

\vskip 5pt

The rest of the proof of this proposition is identical with the proof of \cite[Lem. 2.3]{FSX}. Suppose on the contrary that $\Hom_{G_n}\left(\omega_{n,m}/\omega^0_{n,m},\pi\right) \neq 0$, then there must be an irreducible representation $\sigma$ of $H_m$, such that
\[
    \Hom_{G_n\times H_m}\left(\omega_{n,m}/\omega^0_{n,m},\pi\boxtimes\sigma\right) \neq 0.
\]
By the Howe duality, we must have $\sigma\simeq \theta_{n,m}\left(\pi\right)$. Consider the exact sequence
\[
    \Hom_{G_n\times H_m}\left(\omega_{n,m}/\omega^0_{n,m},\pi\boxtimes\sigma\right) \hookrightarrow \Hom_{G_n\times H_m}\left(\omega_{n,m},\pi\boxtimes\sigma\right) \xrightarrow{\mathscr{R}_\theta} \Hom_{G_n\times H_m}\left(\omega^0_{n,m},\pi\boxtimes\sigma\right).
\]
Here the first arrow is injective by the left exactness of the Hom functor, and the second arrow $\mathscr{R}_\theta$ is non-zero by our claims. It follows that
\[
    \dim\Hom_{G_n\times H_m}\left(\omega_{n,m},\pi\boxtimes\sigma\right) \geq 2,
\]
which is ridiculous since it violates the multiplicity one result Theorem \ref{T:HoweDuality}(1). Therefore we must have $\Hom_{G_n}\left(\omega_{n,m}/\omega^0_{n,m},\pi\right) = 0$ as desired.

\end{proof}

\vskip 5pt

\begin{Rmk}
The advantage of this zeta integral argument we employed is that it can be adapted to the Archimedean case, see the next subsection for a brief discussion. In fact, this proposition can also be proved in a more straightforward way by using the rank filtration of the Weil representation (Lemma \ref{L:rank-filtration}). Suppose that
\[
    \pi = LQ\left(\St\left(\Delta_1\right)\times\St\left(\Delta_2\right)\times\cdots\times\St\left(\Delta_r\right)\right),
\]
with $\Delta_i=\left[x_i,y_i\right]_{\rho_i}$ for $1\leq i \leq r$. Applying the MVW-contragredient trick, we have 
\[
    \pi \hookrightarrow \St\left(\Delta_r\right)\times\St\left(\Delta_{r-1}\right)\times\cdots\times\St\left(\Delta_1\right).
\]
For each $0\leq k \leq n-1$, the graded piece $\tau_k$ of the rank filtration contributes
\begin{align*}
    \Hom_{G_n}\left(\tau_k,\pi\right)_{sm}  \hookrightarrow &\Hom_{G_n}\left(\tau_k,\St\left(\Delta_r\right)\times\St\left(\Delta_{r-1}\right)\times\cdots\times\St\left(\Delta_1\right)\right)_{sm} \\ 
     \simeq &\Ind_{\overline{Q}_k}^{H_m}\left(\Hom_{G_k\times G_{n-k}}\left(\omega_k^\natural\otimes\xi_k, \Jac_{\overline{P}_k}\left(\St\left(\Delta_r\right)\times\cdots\times\St\left(\Delta_1\right)\right)\right)_{sm}\right).
\end{align*}
By the geometric lemma, we know that the Jacquet module on the right hand side admits a filtration, such that each graded piece is of the form $\delta_\alpha\boxtimes \kappa_\alpha$, where $\alpha=\left(\alpha_1,\alpha_2,\cdots,\alpha_r\right)$ is an ordered partition of $n-k$, and
\[
    \kappa_\alpha = \St\left(\left[x_r,x_r-\alpha_r+1\right]_{\rho_r}\right)\times\cdots \times \St\left(\left[x_2,x_2-\alpha_2+1\right]_{\rho_2}\right)\times \St\left(\left[x_1,x_1-\alpha_1+1\right]_{\rho_1}\right).  
\]
Recall that we have assumed the L-function $L\left(s,\pi\right)$ to be holomorphic at $s=\frac{1+m-n}{2}$, which means that for all $1\leq i \leq r$, either $\rho_i\not\simeq \mathbbm{1}_1$ or $x_i\neq -\frac{1+m-n}{2}$ holds. It follows that
\[
    \Hom_{G_{n-k}}\left(\xi_k, \kappa_\alpha\right) = 0
\]
for all ordered partitions $\alpha$ of $n-k$. Then by the K\"unneth formula Theorem \ref{T:Kunneth}, we have
\[
    \Hom_{G_k\times G_{n-k}}\left(\omega_k^\natural\otimes\xi_k, \Jac_{\overline{P}_k}\left(\St\left(\Delta_r\right)\times\cdots\times\St\left(\Delta_1\right)\right)\right)_{sm} =0.
\]
This implies that $\Hom_{G_n}\left(\tau_k,\pi\right)_{sm} = 0$ for all $0\leq k \leq n-1$, hence the proposition holds.

\end{Rmk}

\vskip 5pt

Now back to the proof of our main result I. Let $\pi$ be an irreducible representation of $G_n$ satisfied the assumption in Theorem \ref{T:Main-1}, namely $L\left(s,\pi\right)$ or $L\left(s,\pi^\vee\right)$ is holomorphic at $s=\frac{1+m-n}{2}$. We have already seen in Proposition \ref{P:Holo-L-vanish-Ext} that the Ext-vanishing assertion in Theorem \ref{T:Main-1} holds. To complete the proof of Theorem \ref{T:Main-1} it only remains to compute $\Theta_{n,m}\left(\pi\right)$. We first assume that $L\left(s,\pi\right)$ is holomorphic at $s=\frac{1+m-n}{2}$. Consider the short exact sequence of $G_n\times H_m$ representations
\[
    0 \lra \omega^0_{n,m} \lra \omega_{n,m} \lra  \omega_{n,m}/\omega^0_{n,m} \lra 0.
\]
Applying the functor $\Hom_{G_n}\left(-,\pi\right)_{sm}$ we obtain the exact sequence
\[
    0 \lra \Hom_{G_n}\left(\omega_{n,m}/\omega^0_{n,m},\pi\right)_{sm} \lra \Hom_{G_n}\left(\omega_{n,m},\pi\right)_{sm} \xrightarrow{~\mathscr{R}~} \Hom_{G_n}\left(\omega^0_{n,m},\pi\right)_{sm}.
\]
By Proposition \ref{P:Holo-L-vanish-boundary}, the contribution of the boundary $\Hom_{G_n}\left(\omega_{n,m}/\omega^0_{n,m},\pi\right)_{sm}$ is zero. Hence the map $\mathscr{R}$ is an injection. Moreover, by Proposition \ref{P:Holo-L-vanish-Ext} and Theorem \ref{T:APS-EP-fomula} we know that $\Hom_{G_n}\left(\omega_{n,m},\pi\right)_{sm}$ and $\Hom_{G_n}\left(\omega^0_{n,m},\pi\right)_{sm}$ have the same semi-simplification. Therefore the map $\mathscr{R}$ is an isomorphism, and
\[
    \Hom_{G_n}\left(\omega_{n,m},\pi\right)_{sm} \xrightarrow[~\mathscr{R}~]{\sim} \Hom_{G_n}\left(\omega^0_{n,m},\pi\right)_{sm} \simeq \mathbbm{1}_{m-n}\times \pi.
\]
In this case the desired formula for $\Theta_{n,m}\left(\pi\right)$ follows from taking contragredients of both sides of above.

\vskip 5pt

Next we assume that $L\left(s,\pi^\vee\right)$ is holomorphic at $s=\frac{1+m-n}{2}$ but $L\left(s,\pi\right)$ is not. Consider the MVW-involution of $\pi$. It follows from Corollary \ref{C:ExtvsMVW} that in this case
\[
    \Theta_{n,m}\left(\pi\right) \simeq \Theta_{n,m}\left(\pi^\vee\right)^c\simeq \pi^\vee\times\mathbbm{1}_{m-n}
\]
as desired. This completes the proof of Theorem \ref{T:Main-1}.

\vskip 5pt


\subsection{Digression: the Archimedean case} 
\label{sub:digression_the_archimedean_case}

In this subsection we briefly discuss on extending Theorem 1.4 to the Archimedean case. So in contrast with the main body of this paper, only in this subsection we assume that $F$ is an Archimedean local field, i.e. $F=\R$ or $\C$, and we will work in the setting of \textit{Schwartz analysis} as in \cite{MR4211018}. In particular, by ``representations'' of a linear Nash group, what we actually mean is ``smooth Fr\'echet representations of moderate growth''.

\vskip 5pt

Under this setting our Weil representation is now defined on the space of Schwartz functions $\calS\left(M_{n,m}\right)$. Thanks to \cite{MR985172}, the Howe duality Theorem \ref{T:HoweDuality}(1) is still valid. Recall that in the proof of Theorem \ref{T:Main-1} we have made use of two tools: 

\begin{itemize}
    \item the Jacquet module: some arguments in our proof involves the computation of Jacquet modules of some specific representations;

    \vskip 5pt

    \item the zeta integral: we use the zeta integral to explicitly construct a non-zero element in certain Hom space.  
\end{itemize}
\vskip 5pt
While the Jacquet module becomes much more complicated in the Archimedean case, the theory of zeta integral is still available \cite[Thm. 8.7]{MR342495}. Therefore we can play the same game as in Section \ref{sub:zeta_integrals_in_the_non_equal_rank_case}. More precisely, we have the following analogy of Lemma \ref{L:construction-non-equal-rank-Zeta}.

\begin{Lem}
Suppose that $n\leq m$. Let $\pi$ be an irreducible representation of $G_n$. Then there exists a family of continuous map
\[
    Z^\natural_s : \omega_{n,m} \lra \pi\boxtimes\pi^\vee,
\] 
which is holomorphic in $s\in\C$, and satisfies the following conditions:

\begin{itemize}
     \item when the real part of $s$ is sufficiently large, then $Z^\natural_s$ is given by
     \[
        Z^\natural\left(s,\varphi,f\right) = \frac{Z\left(s,\varphi^\flat,f\right)}{L\left(s,\pi\right)},
     \]
     where $\varphi\in\omega_{n,m}$, $\varphi^\flat$ is defined in the same manner as (\ref{E:Schwartzfunction-flat}), and $f$ is a matrix coefficient of $\pi$;

     \vskip 5pt

     \item for each $s\in\C$, the map $Z_s^\natural$ gives rise to a non-zero element in the Hom space
     \[
        \Hom_{G_n\times\overline{Q}_n}\left(\omega_{n,m}, \pi|\det|^{s-\frac{1+m-n}{2}}\boxtimes\left(\pi^\vee|\det|^{-s+\frac{1}{2}}\boxtimes|\det|^{\frac{n}{2}}\right)\right),
     \]
     where $Q_n$ is the standard parabolic subgroup of $H_m$ with the Levi component $H_n\times H_{m-n}$, and $\overline{Q}_n$ is the opposite parabolic subgroup.
\end{itemize} 
\end{Lem}
\vskip 5pt
By the Frobenius reciprocity, we can regard $Z^\natural_s$ as an non-zero element in
\[
     \Hom_{G_n\times H_m}\left(\omega_{n,m}, \pi|\det|^{s-\frac{1+m-n}{2}}\boxtimes \Ind_{\overline{Q}_n}^{H_m}\left(\pi^\vee|\det|^{-s+\frac{1+m-n}{2}}\boxtimes\mathbbm{1}_{m-n}\right)\right).
\]
The most favorable case would be that $\pi$ is \textit{unitary}. In this case, $\Ind_{\overline{Q}_n}^{H_m}\left(\pi^\vee\boxtimes\mathbbm{1}_{m-n}\right)\simeq \mathbbm{1}_{m-n}\times \pi^\vee$ is irreducible, and the argument in Proposition \ref{P:Holo-L-vanish-boundary} (or, \cite[Lem. 2.3]{FSX}) still works with some mild modifications. Let $\Omega_n$ be the set of matrices in $M_{n,m}$ of rank $n$, and denote by $\omega^0_{n,m}$ the subrepresentation of $\omega_{n,m}$ with the underlying vector space $\calS\left(\Omega_n\right)$. We conclude that:

\begin{Thm}
In the context of this subsection, suppose that $n\leq m$. Let $\pi$ be an irreducible unitary representation of $G_n$. If $L\left(s,\pi\right)$ or $L\left(s,\pi^\vee\right)$ is holomorphic at $s=\frac{1+m-n}{2}$, then 
\[
    \Hom_{G_n}\left(\omega_{n,m}/\omega^0_{n,m},\pi\right) = 0,
\]
and moreover the big theta lift
\[
    \Theta_{n,m}\left(\pi\right) \simeq \pi^\vee \times \mathbbm{1}_{m-n}
\]
is irreducible.
\end{Thm}

\vskip 5pt

We highlight that if $\pi$ is irreducible, unitary and tempered, then $L\left(s,\pi\right)$ is holomorphic in the right half plane ${\rm Re}\left(s\right)>0$. Thus the theorem above holds for $\pi$.


\vskip 10pt


\section{Proof of the main theorem II} 
\label{sec:proof_of_the_main_theorem_ii}

This section is devoted to the proof of the main result II: Theorem \ref{T:Main-2}, and also the complementary result Theorem \ref{T:Main-2-complement}.

\vskip 5pt

\subsection{Piece-wise computations} 
\label{sub:piece_wise_computations}

We first compute contributions of each piece in the rank filtration of the Weil representation. We shall retain the notations in Section \ref{sub:the_rank_filtration}: recall that for each $0\leq k\leq \min\left\{n,m\right\}$, $\Omega_k$ is the set of matrices in $M_{n,m}$ of rank $k$, and the action of $G_n\times H_m$ on $\calS\left(\Omega_k\right)$ inherited from $\omega_{n,m}$ is denoted by $\tau_k$. Lemma \ref{L:rank-filtration} gives a description of this piece
\[
    \tau_k \simeq \Ind_{P_k\times \overline{Q}_k}^{G_n\times H_m} \left(\omega^\natural_k\otimes\xi_k\right).
\]

\begin{Lem}\label{L:ExtGradedPiecesRankFil}
Let $\eta=|\det|^x$ be a character of $G_n$, where $x\in\frac{1}{2}\Z$. Then the Ext spaces associated to each graded piece of the rank filtration and $\eta$ are as follows.
\begin{enumerate}
    \item Suppose that $0\leq k < \min\left\{n,m\right\}$. Then
\[
    \Ext_{G_n}^*\left(\tau_k,\eta\right)_{sm} \simeq 
    \begin{cases}
        \Ind_{\overline{Q}_k}^{H_m}\left(|\det|^{\frac{n-m+k}{2}}\boxtimes |\det|^{\frac{k-n}{2}}\right) \quad & \textit{if } x=k-\frac{m}{2} \textit{ and } * = 0, 1; \\[10pt]

       0  \quad & \textit{otherwise}.
     \end{cases}
\]

\vskip 5pt

\item Suppose that $k=\min\left\{n,m\right\}$. Then 
\[
    \Ext_{G_n}^*\left(\tau_k,\eta\right)_{sm} \simeq 
    \begin{cases}
        \Ind_{\overline{Q}_n}^{H_m}\left(|\det|^x\boxtimes\mathbbm{1}_{m-n}\right) \quad & \textit{if } n\leq m \textit{ and } * = 0; \\[10pt]

        |\det|^{\frac{n}{2}} \quad & \textit{if } n>m \textit{, } x=\frac{m}{2} \textit{ and } * = 0, 1; \\[10pt]

       0  \quad & \textit{otherwise}.
     \end{cases}
\]
\end{enumerate}
\end{Lem}

\vskip 5pt

\begin{proof}
We first prove (1). For each $*\geq 0$, we have
\begin{align*}
    \Ext_{G_n}^*\left(\tau_k,\eta\right)_{sm} & \simeq \Ind_{\overline{Q}_k}^{H_m} \Ext^*_{G_n}\left(\Ind_{P_k}^{G_n}\left(\omega^\natural_k\otimes\xi_k\right) , \eta\right)_{sm} \\
    & \simeq \Ind_{\overline{Q}_k}^{H_m} \Ext^*_{G_k\times G_{n-k}}\left(\omega^\natural_k\otimes\xi_k , |\det|^{x+\frac{n-k}{2}}\boxtimes |\det|^{x-\frac{k}{2}}\right)_{sm}.
\end{align*}
Here in the last equality, we have made use of the second adjointness (see Lemma \ref{L:IndJacAdj}(2)) and the fact that $\Jac_{\overline{P}_k} \eta = |\det|^{x+\frac{n-k}{2}}\boxtimes |\det|^{x-\frac{k}{2}}$. If $k<n$, then by comparing the central characters of $G_{n-k}$-actions on $\xi_k$ and $|\det|^{x-\frac{k}{2}}$, one can see immediately that the Ext space above vanishes unless
\[
    |\det|^{\frac{k-m}{2}} = |\det|^{x-\frac{k}{2}},
\]
or equivalently $x=k-\frac{m}{2}$. In this case, by the K\"unneth formula Theorem \ref{T:Kunneth} and the projectivity of $\omega_k^\natural$ (as a $G_k$-module), we have  
\[
    \Ind_{\overline{Q}_k}^{H_m} \Ext^*_{G_k\times G_{n-k}}\left(\omega^\natural_k\otimes\xi_k , |\det|^{x+\frac{n-k}{2}}\boxtimes |\det|^{x-\frac{k}{2}}\right)_{sm}
\]
\[
    \simeq \Ind_{\overline{Q}_k}^{H_m} \left(\Hom_{G_k}\left(\omega^\natural_k\otimes\xi_k , |\det|^{x+\frac{n-k}{2}}\right)_{sm}\otimes \Ext^*_{G_{n-k}}\left(|\det|^{x-\frac{k}{2}},|\det|^{x-\frac{k}{2}}\right)\right).
\]
Then the rest follows from Lemma \ref{L:ExtSteinberg} easily. The proof for (2) is similar and we omit the details here. 

\end{proof}

\vskip 5pt

The computation above shows that the case $n>m$ is extremely simple: there is at most one graded piece can have non-zero Ext spaces with $\eta$. As a direct consequence, we can complete the proof of Theorem \ref{T:Main-2-complement}.

\vskip 5pt

\begin{proof}[Proof of Theorem \ref{T:Main-2-complement}]
If $\eta \neq |\det|^{k-\frac{m}{2}}$ for any integer $0\leq k\leq m$, then by the previous lemma there is no graded piece (of the rank filtration of the Weil representation) has non-zero Ext spaces with $\eta$, thus
\[
    \Ext_{G_n}^*\left(\omega_{n,m},\eta\right)_{sm} = 0.
\] 
If $\eta=|\det|^{k-\frac{m}{2}}$ for some integer $0\leq k\leq m$, then by the previous lemma and a standard argument of homological algebra we have 
\[
    \Ext_{G_n}^*\left(\omega_{n,m},\eta\right)_{sm} \simeq \Ext_{G_n}^*\left(\tau_k,\eta\right)_{sm}.
\]
Note that under the assumption $n>m$, we know that
\[
    \Ind_{\overline{Q}_k}^{H_m}\left(|\det|^{\frac{n-m+k}{2}}\boxtimes |\det|^{\frac{k-n}{2}}\right) \simeq |{\det}_{m-k}|^{\frac{k-n}{2}}\times|{\det}_k|^{\frac{n-m+k}{2}}
\]
is indeed irreducible. The theorem then follows from the previous lemma.

\end{proof}

\vskip 5pt


\subsection{Gluing graded pieces} 
\label{sub:gluing_graded_pieces}

Next we shall focus on the case that $n\leq m$. Thanks to Theorem \ref{T:Main-1}, we only need to consider the case that $n\geq 2$, $m\leq 2n-2$ and those characters $\eta$ of the form 
\[
    \eta = |\det|^{k-\frac{m}{2}}
\]
for some integer $1+m-n \leq k \leq n-1$. The computation in Lemma \ref{L:ExtGradedPiecesRankFil} shows that there are exactly two graded pieces of the rank filtration can have non-zero Ext spaces with $\eta$. The main goal of this subsection is to use some long exact sequences to ``glue'' these Ext spaces associated to graded pieces and $\eta$.

\vskip 5pt

We recall that the rank filtration on $\omega_{n,m}$ is a descending sequence of $G_n\times H_m$-modules
\[
    0 \subset R_n \subset R_{n-1} \subset \cdots \subset R_0 = \omega_{n,m},
\]
and for each $0\leq i\leq n$, we have $\tau_i\simeq R_i/R_{i+1}$. By Lemma \ref{L:ExtGradedPiecesRankFil}, only $\Ext_{G_n}^*\left(\tau_k,\eta\right)_{sm}$ and $\Ext_{G_n}^*\left(\tau_n,\eta\right)_{sm}$ can be non-zero. Therefore it is natural to consider the short exact sequence
\begin{equation}\label{E:SES-glue-rank-fil}
    0 \lra R_{k+1} \lra \omega_{n,m} \lra \omega_{n,m}/R_{k+1} \lra 0,
\end{equation}
since this exact sequence separates $\tau_k$ and $\tau_n$. 

\begin{Lem}
In the context of above discussions, the Ext spaces associated to $R_{k+1}$ and $\eta$ are
\[
    \Ext_{G_n}^*\left(R_{k+1},\eta\right)_{sm} \simeq \begin{cases}
        \mathbbm{1}_{m-n}\times|\det_n|^{k-\frac{m}{2}} \quad & \textit{if } * = 0; \\[10pt]
       0  \quad & \textit{otherwise}.
     \end{cases}
\]   
Likewise, the Ext spaces associated to $\omega_{n,m}/R_{k+1}$ and $\eta$ are
\[
    \Ext_{G_n}^*\left(\omega_{n,m}/R_{k+1},\eta\right)_{sm} \simeq \begin{cases}
        |{\det}_{m-k}|^{\frac{k-n}{2}}\times|{\det}_k|^{\frac{n-m+k}{2}} \quad & \textit{if } * = 0, 1; \\[10pt]
       0  \quad & \textit{otherwise}.
     \end{cases}
\]  
\end{Lem}

\vskip 5pt

\begin{proof}
Again, by appealing to Lemma \ref{L:ExtGradedPiecesRankFil}, we know that 
\[
    \Ext_{G_n}^*\left(R_{k+1},\eta\right)_{sm} \simeq \Ext_{G_n}^*\left(\tau_n,\eta\right)_{sm}.
\]
Note that $\Ind_{\overline{Q}_n}^{H_m}\left(|\det|^{k-\frac{m}{2}}\boxtimes\mathbbm{1}_{m-n}\right)\simeq \mathbbm{1}_{m-n}\times|\det_n|^{k-\frac{m}{2}}$. This implies the desired formula for $\Ext_{G_n}^*\left(R_{k+1},\eta\right)_{sm}$. The computation of $\Ext_{G_n}^*\left(\omega_{n,m}/R_{k+1},\eta\right)_{sm}$ is similar and we shall not repeat.

\end{proof}

\vskip 5pt

Now apply the functor $\Hom_{G_n}\left(-,\eta\right)_{sm}$ to the short exact sequence (\ref{E:SES-glue-rank-fil}), we get a long exact sequence
\begin{flushleft}
$0 \lra |{\det}_{m-k}|^{\frac{k-n}{2}}\times|{\det}_k|^{\frac{n-m+k}{2}} \lra \Hom_{G_n}\left(\omega_{n,m},\eta\right)_{sm} \xrightarrow{~\mathscr{I}_k~} \mathbbm{1}_{m-n}\times|\det_n|^{k-\frac{m}{2}}$
\end{flushleft}
\begin{flushright}
$\lra  |{\det}_{m-k}|^{\frac{k-n}{2}}\times|{\det}_k|^{\frac{n-m+k}{2}}  \lra \Ext^1_{G_n}\left(\omega_{n,m},\eta\right)_{sm} \lra 0,$   
\end{flushright}
and also
\[
    \Ext^*_{G_n}\left(\omega_{n,m},\eta\right)_{sm} = 0
\]
for all $* \geq 2$. Therefore, to prove Theorem \ref{T:Main-2} it only remains to show that the map $\mathscr{I}_k$ is zero. Recall that $1+m-n\leq k \leq n-1$, so the segment corresponding to $\mathbbm{1}_{m-n}$ is strictly contained in the segment corresponding to $|\det_n|^{k-\frac{m}{2}}$. It follows that $\mathbbm{1}_{m-n}\times|\det_n|^{k-\frac{m}{2}}$ is irreducible. Moreover, by Theorem \ref{T:IndSeg}(2), we have 
\[
    {\rm soc}\left(\Hom_{G_n}\left(\omega_{n,m},\eta\right)_{sm}\right) = {\rm soc}\left(|{\det}_{m-k}|^{\frac{k-n}{2}}\times|{\det}_k|^{\frac{n-m+k}{2}}\right) = \mathbbm{1}_{m-n}\times|{\det}_n|^{k-\frac{m}{2}}.
\]
Observe that if we can show that $\Hom_{G_n}\left(\omega_{n,m},\eta\right)_{sm}$ is a socle irreducible representation of $H_m$ (i.e. its socle is irreducible and occurs with multiplicity one), then the map $\mathscr{I}_k$ must be zero; otherwise  
$\mathbbm{1}_{m-n}\times|{\det}_n|^{k-\frac{m}{2}}$
will appear in $\Hom_{G_n}\left(\omega_{n,m},\eta\right)_{sm}$ twice. So to complete the proof of Theorem \ref{T:Main-2} it suffices to show the following lemma.

\vskip 5pt


\begin{Lem}
Suppose that $n\geq 2$, $n\leq m\leq 2n-2$. Let $\eta = |{\det}_n|^{k-\frac{m}{2}}$ be a character of $G_n$, where $1+m-n\leq k\leq n-1$ is a positive integer. Then 
\[
    \Hom_{G_n}\left(\omega_{n,m},\eta\right)_{sm}
\]    
is socle irreducible.
\end{Lem}

\vskip 5pt

\begin{proof}
We shall prove the lemma by induction on $n\in\Z_{>0}$. The basic case is that $n=m=2$, and $\eta=\mathbbm{1}_2$ is the trivial character of $G_2$. In this case, there is an embedding $\eta\hookrightarrow|\cdot|^{-\frac{1}{2}}\times|\cdot|^{\frac{1}{2}}$. Hence
\begin{align*}
 \Hom_{G_2}\left(\omega_{2,2},\eta\right)_{sm} \hookrightarrow &\Hom_{G_2}\left(\omega_{2,2},|\cdot|^{-\frac{1}{2}}\times|\cdot|^{\frac{1}{2}}\right)_{sm} 
 \simeq  |\cdot|^{-\frac{1}{2}}\times |\cdot|^{\frac{1}{2}}.
\end{align*}
Here in the last identity we have made use of Corollary \ref{C:KudlaIndCompatible} and Theorem \ref{T:Main-1}. Since $|\cdot|^{-\frac{1}{2}}\times |\cdot|^{\frac{1}{2}}$ is socle irreducible, we know that $\Hom_{G_2}\left(\omega_{2,2},\eta\right)_{sm}$ is also socle irreducible. 

\vskip 5pt

Suppose now the lemma holds for all $n'<n$, we prove it also holds for $n$. Let $\eta = |{\det}_n|^{k-\frac{m}{2}}$. Similar to the basic case, one can consider the embedding
\[
    \eta \hookrightarrow |\cdot|^{k+\frac{1-n-m}{2}} \times |{\det}_{n-1}|^{k-\frac{m-1}{2}}.
\]
Note that $k+\frac{1-n-m}{2} < \frac{1+m-n}{2}$. Hence we can apply Corollary \ref{C:KudlaIndCompatible}, it follows that
\[
    \Hom_{G_n}\left(\omega_{n,m},\eta\right)_{sm}\hookrightarrow |\cdot|^{k+\frac{1-n-m}{2}} \times \Hom_{G_{n-1}}\left(\omega_{n-1,m-1},|{\det}_{n-1}|^{k-\frac{m-1}{2}}\right)_{sm}.
\]
To show the representation $\Hom_{G_n}\left(\omega_{n,m},\eta\right)_{sm}$ is socle irreducible, it suffices to show that the socle of $\Hom_{G_n}\left(\omega_{n,m},\eta\right)_{sm}$, namely $\mathbbm{1}_{m-n}\times|{\det}_n|^{k-\frac{m}{2}}$, appears in
\[
   \Pi \coloneqq |\cdot|^{k+\frac{1-n-m}{2}} \times \Hom_{G_{n-1}}\left(\omega_{n-1,m-1},|{\det}_{n-1}|^{k-\frac{m-1}{2}}\right)_{sm}
\]
with multiplicity one. We shall do it by computing some Jacquet modules. Let $d\in\Z_{>0}$, and $Q_{1,d-1}$ 
be the standard parabolic subgroup of $H_d$ with Levi component $H_1\times H_{d-1}$. For a finite length representation $\pi$ of $H_d$ and a real number $x\in\R$, if the semi-simplified Jacquet module of $\pi$ along $Q_{1,d-1}$ takes the form
\[
    s.s.\Jac_{Q_{1,d-1}}\pi = \sum_\alpha \kappa_\alpha\boxtimes \sigma_\alpha
\]
for some characters $\kappa_\alpha$ of $H_1$ and irreducible representations $\sigma_\alpha$ of $H_{d-1}$, then we set 
\[
    \Jac_{x}\pi = \sum_{\left\{\alpha\,:\,\kappa_\alpha=|\cdot|^x\right\}} \sigma_\alpha,
\]
which is a finite length semi-simple representation of $H_{d-1}$. Note that $k+\frac{1-n-m}{2} < \frac{1-m+n}{2}$. Hence on the one hand by Proposition \ref{P:JacSeg} and the geometric lemma we have 
\[
     \Jac_{k+\frac{1-n-m}{2}}\left(\mathbbm{1}_{m-n}\times|{\det}_n|^{k-\frac{m}{2}}\right) = \mathbbm{1}_{m-n}\times|{\det}_{n-1}|^{k-\frac{m-1}{2}}
\]
is non-zero and irreducible. On the other hand, again by the geometric lemma, we have 
\begin{flushleft}
$\Jac_{k+\frac{1-n-m}{2}} \Pi = \Hom_{G_{n-1}}\left(\omega_{n-1,m-1},|{\det}_{n-1}|^{k-\frac{m-1}{2}}\right)_{sm} + $
\end{flushleft}
\begin{flushright}
    $|\cdot|^{k+\frac{1-n-m}{2}} \times \Jac_{k+\frac{1-n-m}{2}}\Hom_{G_{n-1}}\left(\omega_{n-1,m-1},|{\det}_{n-1}|^{k-\frac{m-1}{2}}\right)_{sm}.$
\end{flushright}
There are two cases depending on the positive integer $k$, and now we discuss case by case.

\vskip 5pt

$\bullet$ \textit{Case 1}: if $1+m-n\leq k\leq n-2$, then by the induction hypothesis the representation $\Hom_{G_{n-1}}\left(\omega_{n-1,m-1},|{\det}_{n-1}|^{k-\frac{m-1}{2}}\right)_{sm}$ is socle irreducible. Therefore by above discussions, Theorem \ref{T:Main-2} holds for $|{\det}_{n-1}|^{k-\frac{m-1}{2}}$, namely
\[
    \Hom_{G_{n-1}}\left(\omega_{n-1,m-1},|{\det}_{n-1}|^{k-\frac{m-1}{2}}\right)_{sm} \simeq |{\det}_{m-k-1}|^{\frac{k-n+1}{2}}\times|{\det}_k|^{\frac{n-m+k}{2}}.
\]
So once again by Proposition \ref{P:JacSeg} and the geometric lemma we have
\[
    \Jac_{k+\frac{1-n-m}{2}}\Hom_{G_{n-1}}\left(\omega_{n-1,m-1},|{\det}_{n-1}|^{k-\frac{m-1}{2}}\right)_{sm}
\]
\[
    =\left(\Jac_{k+\frac{1-n-m}{2}}|{\det}_{m-k-1}|^{\frac{k-n+1}{2}}\right)\times|{\det}_k|^{\frac{n-m+k}{2}} + |{\det}_{m-k-1}|^{\frac{k-n+1}{2}}\times \left(\Jac_{k+\frac{1-n-m}{2}}|{\det}_k|^{\frac{n-m+k}{2}}\right)
\]
vanishes. Combining these together we get
\[
    \Jac_{k+\frac{1-n-m}{2}} \Pi \simeq |{\det}_{m-k-1}|^{\frac{k-n+1}{2}}\times|{\det}_k|^{\frac{n-m+k}{2}}.
\]
According to Theorem \ref{T:IndSeg}(2), the degenerate principal series on the right hand side of above is of length two, with non-isomorphic sub and quotient. Hence $\mathbbm{1}_{m-n}\times|{\det}_n|^{k-\frac{m}{2}}$ can only appear in $\Pi$ once.

\vskip 5pt

$\bullet$ \textit{Case 2}: if $k=n-1$, then by Proposition \ref{P:GJLfunction} the L-function $L\left(s,|{\det}_{n-1}|^{k-\frac{m-1}{2}}\right)$ is holomorphic at $s=\frac{1+m-n}{2}$. So we are in a situation that Theorem \ref{T:Main-1} is applicable, we have 
\[
    \Hom_{G_{n-1}}\left(\omega_{n-1,m-1},|{\det}_{n-1}|^{k-\frac{m-1}{2}}\right)_{sm} \simeq \mathbbm{1}_{m-n}\times |{\det}_{n-1}|^{k-\frac{m-1}{2}}.
\]
Similar to the previous case, it follows that
\[
    \Jac_{k+\frac{1-n-m}{2}}\Hom_{G_{n-1}}\left(\omega_{n-1,m-1},|{\det}_{n-1}|^{k-\frac{m-1}{2}}\right)_{sm}
\]
\[
    =\left(\Jac_{k+\frac{1-n-m}{2}}\mathbbm{1}_{m-n}\right)\times|{\det}_{n-1}|^{k-\frac{m-1}{2}} + \mathbbm{1}_{m-n}\times \left(\Jac_{k+\frac{1-n-m}{2}}|{\det}_{n-1}|^{k-\frac{m-1}{2}}\right)
\]
vanishes. Combining these together we get
\[
    \Jac_{k+\frac{1-n-m}{2}} \Pi \simeq \mathbbm{1}_{m-n}\times |{\det}_{n-1}|^{k-\frac{m-1}{2}}.
\]
According to Theorem \ref{T:IndSeg}, the degenerate principal series on the right hand side of above is either irreducible, or of length two with non-isomorphic sub and quotient. Hence $\mathbbm{1}_{m-n}\times|{\det}_n|^{k-\frac{m}{2}}$ can only appear in $\Pi$ once.

\vskip 5pt

This shows that the lemma also holds for $n$, hence completes the proof.

\end{proof}

\vskip 5pt

As we have explicated, the vanishing of the map $\mathscr{I}_k$ is a direct consequence of this lemma, and Theorem \ref{T:Main-2} then follows from the vanishing of $\mathscr{I}_k$.

\vskip 5pt

\begin{Rmk}
In the special case $n=m=2$ and $k=1$, Xue showed this map $\mathscr{I}_1=0$ \cite[Prop. 2.9]{xue2023full} using a completely different way. His idea is to consider the Haar measure (twisted by determinant) $D$ on $G_2$, which gives rise to an invariant distribution on $\Omega_2$. Given this $D$, he constructed a particular extension $\widetilde{D}$ to $M_{2,2}$ by considering the constant term of the Laurent expansion of the zeta integral. This $\widetilde{D}$ is generalized semi-invariant but not invariant, because of the pole of the zeta integral. Then he showed that there is no invariant extension of $D$ using this particular extension $\widetilde{D}$: suppose such an invariant extension $\widetilde{D}'$ exists, then 
\[
    \widetilde{D}- \widetilde{D}'
\]
yields a generalized semi-invariant distribution supported on $M_{2,2}\backslash \Omega_2$; but on $M_{2,2}\backslash \Omega_2$ generalized semi-invariance implies invariance, so $\widetilde{D} = \widetilde{D}' + \left(\widetilde{D} - \widetilde{D}'\right)$ is also invariant, this leads to a contradiction. This implies that $D$ is not in the image of $\mathscr{I}_k$, so $\mathscr{I}_k=0$.

\vskip 5pt

It would be interesting if one can adapt Xue's argument to $2\leq n\leq m$, or to the Archimedean case.

\end{Rmk}


\vskip 10pt


\section{Proof of the main theorem III} 
\label{sec:proof_of_the_main_theorem_iii}

Throughout this section we assume that $n\leq m$. Our main purpose is to prove the main result III: Theorem \ref{T:Main-3}. By Theorem \ref{T:Main-2}, we only need to show that if $m\geq 2n-1$, then $\omega_{n,m}$ is projective. 
After showing this, we will give an upper bound of the \textit{projective dimension} of the $K$-fixed part $\omega_{n,m}^K$ of the Weil representation when $n\leq m\leq 2n-2$, where $K$ is an open compact subgroup of $H_m$.

\vskip 5pt

\subsection{Locally finiteness of the Weil representation} 
\label{sub:locally_finiteness_of_the_weil_representation}

We first show that, given an open compact subgroup $K$ of $H_m$, then the ``level $K$'' Weil representation, i.e. $\omega_{n,m}^K$, is a so called \textit{locally finite} representation of $G_n$, as defined in \cite[Def. 5.2]{MR4054816}. The tool that we shall use is the rank filtration (Lemma \ref{L:rank-filtration}). By the definition of locally finite representations, it suffices to show that (the ``level $K$'' part of) each graded piece of the rank filtration
\[
    \tau_i^K \simeq \left(\Ind_{P_i\times \overline{Q}_i}^{G_n\times H_m} \left(\omega^\natural_i\otimes\xi_i\right)\right)^K
\]
is actually a locally finite representation of $G_n$. 
Since here we are mainly concerned about the $G_n$-module structure, we can rewrite the right hand side of above as 
\begin{equation}\label{E:LocallyFiniteLevelK}
    \left(\Ind_{P_i\times \overline{Q}_i}^{G_n\times H_m} \left(\omega^\natural_i\otimes\xi_i\right)\right)^K \simeq \bigoplus_{x\,\in \,\overline{Q}_i\backslash H_m /K} \left(\Ind_{P_i}^{G_n}\left(\omega^\natural_i\otimes\xi_i\right)\right)^{\overline{Q}_i\,\cap\, x\cdot K\cdot x^{-1}}.
\end{equation}
Here the isomorphism is given by sending a function
\[
    f\in \left(\Ind_{P_i\times \overline{Q}_i}^{G_n\times H_m} \left(\omega^\natural_i\otimes\xi_i\right)\right)^K \simeq \left(\Ind_{\overline{Q}_i}^{H_m}\left(\Ind_{P_i}^{G_n}\left(\omega^\natural_i\otimes\xi_i\right)\right)\right)^K
\]
to its values at a set of representatives of the double coset $\overline{Q}_i\backslash H_m /K$. Note that the double coset $\overline{Q}_i\backslash H_m /K$ is finite. Therefore we only need to show that each summand appearing on the right hand side of (\ref{E:LocallyFiniteLevelK}) is locally finite. Let $K_x$ be an open compact subgroup of $H_i$, such that $K_x\subset \overline{Q}_i\cap x\cdot K\cdot x^{-1}$. Then
\[
    \left(\Ind_{P_i}^{G_n}\left(\omega^\natural_i\otimes\xi_i\right)\right)^{\overline{Q}_i\,\cap\, x\cdot K\cdot x^{-1}} \subset \Ind_{P_i}^{G_n}\left(\left(\omega^\natural_i\right)^{K_x}\otimes\xi_i\right).
\]
Easy to see that $\left(\omega^\natural_i\right)^{K_x}\otimes\xi_i$ is a locally finite representation of the Levi subgroup $G_i\times G_{n-i}$ of $P_i$. Since the parabolic induction functor preserves locally finiteness, we know that the representation $\Ind_{P_i}^{G_n}\left(\left(\omega^\natural_i\right)^{K_x}\otimes\xi_i\right)$ is locally finite. Therefore by the Noetherian property of the category $\calM\left(G\right)$, its subrepresentation $\left(\Ind_{P_i}^{G_n}\left(\omega^\natural_i\otimes\xi_i\right)\right)^{\overline{Q}_i\,\cap\, x\cdot K\cdot x^{-1}}$ is also locally finite.

\vskip 5pt

This completes the proof of the locally finiteness of $\omega_{n,m}^K$.

\vskip 5pt


\subsection{Zelevinsky's classification and Ext-vanishing} 
\label{sub:zelevinsky_s_classification_and_ext_vanishing}

Secondly, we shall prove that the ``level $K$'' part of the Weil representation $\omega_{n,m}^K$ is projective as a representation of $G_n$ when $m\geq 2n-1$. Since we already know that $\omega_{n,m}^K$ is locally finite, by the result of Chan--Savin \cite[Thm. A.1]{MR3910471}, it suffices to show that 
\begin{equation}\label{E:ExtVanishingLevelK}
    \Ext_{G_n}^*\left(\omega_{n,m}^K,\pi\right) = 0
\end{equation}
for all irreducible smooth representation $\pi$ of $G_n$ and all $* >0$. By the proof of \cite[Lem. 5.14]{MR3753906}, one knows that
\[
    \Ext_{G_n}^*\left(\omega_{n,m}^K,\pi\right) \simeq \Ext_{G_n}^*\left(\omega_{n,m},\pi\right)^K.
\]
Therefore the desired Ext-vanishing result (\ref{E:ExtVanishingLevelK}) follows from the following lemma.

\begin{Prop}\label{P:AlmostStableRangeExtVanishing}
Suppose that $m\geq 2n-1$. Then for any irreducible representation $\pi$ of $G_n$ and any $*>0$, we have
\[
    \Ext_{G_n}^*\left(\omega_{n,m},\pi\right)_{sm} = 0.
\]  
\end{Prop}

\vskip 5pt

\begin{proof}
We shall prove this proposition by induction on $n$. The desired conclusion for $n=1$ is already shown in \cite[Lem. 6.1]{MR3753906}, so we assume that $n\geq 2$, and the desired conclusion holds for all $n'<n$. Given an irreducible representation $\pi$ of $G_n$. If $\pi$ is supercuspidal, consider the rank filtration of the Weil representation (see Lemma \ref{L:rank-filtration}). For any $0\leq i <n$ and any $*\geq 0$, by the supercuspidality of $\pi$ we have
\[
    \Ext_{G_n}^*\left(\tau_i, \pi\right)_{sm} \simeq \Ind_{\overline{Q}_i}^{H_m} \Ext_{G_i\times G_{n-i}}^*\left(\omega_{i}^\natural\otimes\xi_i, \Jac_{\overline{P}_i}\pi\right)_{sm} = 0.
\]
Therefore it follows that
\[
    \Ext_{G_n}^*\left(\tau_i, \pi\right)_{sm} \simeq \Ind_{\overline{Q}_n}^{H_m}\left(\Ext_{G_n}^*\left(\omega_n^\natural, \pi\right)_{sm}\otimes\xi_n\right)
\]
for all $*>0$. If $\pi$ is not supercuspidal, then by the \textit{Zelevinsky's classification} \cite[Thm. 6.1]{ZeleII}, we can write $\pi$ as a quotient
\[
    \Speh\left(\Delta_1\right) \times \Speh\left(\Delta_2\right) \times \cdots \times \Speh\left(\Delta_r\right) \twoheadrightarrow \pi,  
\]
where $\Delta_i = \left[x_i,y_i\right]_{\rho_i}$, and $x_1+y_1 \leq x_2+y_2 \leq \cdots \leq x_r+y_r$. Note that by Corollary \ref{C:ExtvsMVW}, without loss of generality we can assume that the central exponent of $\pi$ is non-positive. Under this assumption we have
\[
    x_1+ y_1\leq 0.
\]
On the other hand, suppose that $\rho_i$ is an irreducible supercuspidal representation of $G_{k_i}$, then $G_{k_1(x_1-y_1+1)}\times G_{k_2(x_2-y_2+1)} \times \cdots\times G_{k_r(x_r-y_r+1)}$ is a Levi subgroup of $G_n$, in particular
\[
    k_1\left(x_1-y_1+1\right) \leq n.
\]
These two inequalities, together with the assumption $m\geq 2n-1$, imply that
\[
    x_1 \leq \frac{n-1}{2} < \frac{1+m-n}{2}.
\]
Now let $\Pi = \Speh\left(\left[x_1-1,y_1\right]_{\rho_1}\right) \times \Speh\left(\Delta_2\right) \times \cdots \times \Speh\left(\Delta_r\right)$, so $\pi$ is a quotient of $\rho_1|\det|^{x_1}\times\Pi$, and we complete this quotient map to a short exact sequence  
\[
    0 \lra \varkappa \lra \rho_1|\det|^{x_1}\times\Pi \lra \pi \lra 0.
\]
Here $\varkappa$ is a finite length representation of $G_n$. Applying the functor $\Hom_{G_n}\left(\omega_{n,m},-\right)$ to this short exact sequence, we get 
\[
    \cdots\ra\Ext_{G_n}^*\left(\omega_{n,m},\rho_1|\det|^{x_1}\times\Pi\right)_{sm} \lra \Ext_{G_n}^*\left(\omega_{n,m},\pi\right)_{sm} \lra \Ext_{G_n}^{*+1}\left(\omega_{n,m},\varkappa\right)_{sm}\ra\cdots
\]
Since $x_1 < \frac{1+m-n}{2}$, we may invoke Corollary \ref{C:KudlaIndCompatible}, which gives
\[
    \Ext_{G_n}^*\left(\omega_{n,m},\rho_1|\det|^{x_1}\times\Pi\right)_{sm} \simeq \rho_1|\det|^{x_1}\times \Ext_{G_n}^*\left(\omega_{n-k_1,m-k_1},\Pi\right)_{sm} = 0.
\]
Here in the second equality, we have made use of our induction hypothesis. Therefore in order to prove the proposition it only remains to show that for any irreducible representation $\pi$, we have 
\begin{equation*}
    \Ext_{G_n}^{*+1}\left(\omega_{n,m},\pi\right)_{sm} = 0.
\end{equation*}
Repeating this argument, eventually we can reduce the desired conclusion to showing that
\begin{equation*}
    \Ext_{G_n}^{*+n}\left(\omega_{n,m},\pi\right)_{sm} = 0
\end{equation*}
for all irreducible representation $\pi$, which simply follows from \cite[Pg. 98, Sect. 4.2]{BNote}. Hence the proposition also holds for $n$.

\end{proof}

\vskip 5pt

Thirdly, we pass from the ``level $K$'' part $\omega_{n,m}^K$ to the whole $\omega_{n,m}$. Fix an open compact subgroup $K$ of $H_m$. Then we can decompose $\omega_{n,m}$ into a direct sum
\[
    \omega_{n,m} = \bigoplus_{\sigma\in\Irr\left(K\right)} \omega_\sigma,
\]
where each $\omega_\sigma \simeq \Hom_{K}\left(\sigma,\omega_{n,m}\right)\boxtimes\sigma$ is the $\sigma$-isotypic component of $\omega_{n,m}$; it is a $G_n\times K$-submodule of $\omega_{n,m}$. For each $\sigma\in\Irr\left(K\right)$, since $\sigma$ is finite dimensional, there exists an open compact subgroup $K_0$ of $K$, such that all vectors in $\omega_\sigma$ is fixed by $K_0$. It follows that $\omega_\sigma$ is a direct summand of
\[
    \omega_{n,m}^{K_0} = \bigoplus_{\sigma'\in\Irr\left(K\right)} \omega_{\sigma'}^{K_0}.
\]
Since we have proved that $\omega_{n,m}^{K_0}$ is a projective $G_n$-module, its direct summand $\omega_\sigma$ is also a projective $G_n$-module. Therefore, being a direct sum of projective modules $\omega_\sigma$, $\omega_{n,m}$ is itself projective.

\vskip 5pt


\subsection{Projective dimension of the Weil representation} 
\label{sub:projective_dimension_of_the_weil_representation}

In this subsection, we assume that $m\leq 2n-2$; so by Theorem \ref{T:Main-2}, the Weil representation $\omega_{n,m}$ is \textit{not} projective as a $G_n$-module. As usual, given an open compact subgroup $K$ of $H_m$, one can consider the projective dimension of $\omega_{n,m}^K$: if 
\[
    0 \lra \calP_d \lra \calP_{d-1} \lra \cdots \lra \calP_1 \lra\calP_0 \lra 0
\]
is a projective resolution of $\omega_{n,m}^K$ in the category $\calM\left(G_n\right)$ of minimal length, then one sets
\[
    {\rm proj.}\dim \omega_{n,m}^K = d.
\]
Likewise, one can also consider projective dimension of $\omega_{n,m}$ itself, which is defined in the same manner. The advantage of working with $\omega_{n,m}^K$ instead of $\omega_{n,m}$ is that $\omega_{n,m}^K$, as we have already seen in Section \ref{sub:locally_finiteness_of_the_weil_representation}, is a locally finite representation of $G_n$. The main purpose of this subsection is to prove the following result.

\begin{Prop}\label{P:BoundExtVanishingDegree}
Suppose that $n\geq 2$, and $n\leq m\leq 2n-2$. Then for any open compact subgroup $K$ of $H_m$, we have
\[
    {\rm proj.}\dim \omega_{n,m}^K \leq \left\lceil n - \frac{m+1}{2}\right\rceil.
\]  
Here for $x\in\R$, we denote by $\left\lceil x \right\rceil$ the smallest integer greater than or equal to $x$. 
\end{Prop}

\vskip 5pt

\begin{proof}
Since $\omega_{n,m}^K$ is locally finite, by some elementary homological algebra argument, we know that ${\rm proj.}\dim \omega_{n,m}^K \leq \left\lceil n - \frac{m+1}{2}\right\rceil$ if and only if 
\[
    \Ext_{G_n}^*\left(\omega_{n,m}^K,\pi\right) = 0
\]
holds for all irreducible representation $\pi$ of $G_n$ and all $*>\left\lceil n - \frac{m+1}{2}\right\rceil$. Let $K$ varies. Then these Ext spaces vanish for all $K$ if and only if
\begin{equation}\label{E:BoundingProjDim}
    \Ext_{G_n}^*\left(\omega_{n,m},\pi\right)_{sm} = 0.
\end{equation}
So next we shall prove (\ref{E:BoundingProjDim}) by induction on $n$.

\vskip 5pt

The basic case is when $n=m=2$. In this case, one can check that there is only one representation $\pi$ of $G_2$ such that both $L\left(s,\pi\right)$ and $L\left(s,\pi^\vee\right)$ are not holomorphic at $s=1/2$, namely the trivial representation $\mathbbm{1}_2$. Hence for non-trivial irreducible representation $\pi$ of $G_2$, by Theorem \ref{T:Main-1} all these Ext spaces vanish when $*>0$. Moreover, all Ext spaces associated to $\omega_{2,2}$ and $\mathbbm{1}_2$ have been computed in Theorem \ref{T:Main-2}. Hence the desired conclusion can be checked by hand.

\vskip 5pt

Suppose now the desired conclusion (\ref{E:BoundingProjDim}) has been proved for all $n'<n$. Let $\pi$ be an irreducible representation of $G_n$. If $\pi$ is supercuspidal, then by the same argument as in the proof of Proposition \ref{P:AlmostStableRangeExtVanishing}, we know that (\ref{E:BoundingProjDim}) holds for $\pi$. If $\pi$ is not supercuspidal, then there exists a positive integer $0<k<n$, an irreducible supercuspidal representation $\rho$ of $G_k$ and an irreducible representation $\pi_0$ of $G_{n-k}$, such that there is a surjection
\[
    \rho\times\pi_0 \twoheadrightarrow \pi.
\]
By a standard degree shifting argument, to show that (\ref{E:BoundingProjDim}) holds for $\pi$, it suffices to show that
\begin{equation}\label{E:BoundingExtDegIndRep}
    \Ext_{G_n}^*\left(\omega_{n,m}, \rho\times\pi_0\right)_{sm} = 0
\end{equation}
for all $*>\left\lceil n - \frac{m+1}{2}\right\rceil$. If $k\neq 1$, or $\rho\not\simeq|\cdot|^{\frac{1+m-n}{2}}$, then the desired conclusion follows from the combination of Corollary \ref{C:KudlaIndCompatible} and the induction hypothesis. On the other hand, if $k=1$ and $\rho\simeq|\cdot|^{\frac{1+m-n}{2}}$, we shall appeal to the Kudla's filtration directly. Let $P_1$ be the standard parabolic subgroup of $G_n$ with Levi component $G_1\times G_{n-1}$. Then by the Frobenius reciprocity we have
\[
    \Ext_{G_n}^*\left(\omega_{n,m}, \rho\times\pi_0\right)_{sm} \simeq \Ext_{G_1\times G_{n-1}}^*\left(\Jac_{P_1}\omega_{n,m}, \rho\boxtimes\pi_0\right)_{sm}.
\]
By Proposition \ref{P:KudlaFiltration}, there is a two-step filtration on $\Jac_{P_1}\omega_{n,m}$:
\[
    0 = J_2 \subset J_1 \subset J_0 = \Jac_{P_1}\omega_{n,m},
\]
whose successive quotient $\lambda_i = J_i/J_{i+1}$ can be described as follows:
\[
    \lambda_0 \simeq |\cdot|^{\frac{1+m-n}{2}} \boxtimes \omega_{n-1,m},
\]
and 
\[
    \lambda_1 \simeq \Ind_{Q_1}^{H_m}\left(\omega_1^\natural\boxtimes\omega_{n-1,m-1}\right).
\]
Applying the functor $\Hom_{G_n}\left(\omega_{n,m},-\right)_{sm}$ to this short exact sequence
\[
    0 \lra \lambda_1 \lra \Jac_{P_1}\omega_{n,m} \lra \lambda_0 \lra 0,
\]
we obtain the long exact sequence
\[
    \cdots\ra\Ext^*_{G_n}\left(\lambda_0, \rho\boxtimes\pi_0\right)_{sm} \lra \Ext^*_{G_n}\left(\Jac_{P_1}\omega_{n,m}, \rho\boxtimes\pi_0\right)_{sm} \lra \Ext^{*}_{G_n}\left(\lambda_1, \rho\boxtimes\pi_0\right)_{sm}\ra\cdots
\]
By the K\"unneth formula Theorem \ref{T:Kunneth}, we have
\begin{align*}
 \Ext^*_{G_n}\left(\lambda_0, \rho\boxtimes\pi_0\right)_{sm} = \Ext_{G_{n-1}}^{*-1}\left(\omega_{n-1,m},\pi_0\right)_{sm} \oplus \Ext_{G_{n-1}}^{*}\left(\omega_{n-1,m},\pi_0\right)_{sm}.
\end{align*}
Both two Ext spaces on the right hand side of above vanish when $*>\left\lceil n - \frac{m+1}{2}\right\rceil$ by the induction hypothesis. Moreover, we always have
\[
    \Ext^{*}_{G_n}\left(\lambda_1, \rho\boxtimes\pi_0\right)_{sm} \simeq \rho\times \Ext^{*}_{G_{n-1}}\left(\omega_{n-1,m-1}, \pi_0\right)_{sm},
\]
which also vanishes when $*>\left\lceil n - \frac{m+1}{2}\right\rceil$ by the induction hypothesis. Therefore (\ref{E:BoundingExtDegIndRep}) holds, and it follows that (\ref{E:BoundingProjDim}) holds for $n$. This completes the proof of the proposition.

\end{proof}

\vskip 5pt

Unfortunately, this bound is not sharp. For example, when $n=m=4$, $\left\lceil n - \frac{m+1}{2}\right\rceil=2$. However, with a bit more effort, one can show that
\begin{equation}\label{E:G4Ext2vanish}
    \Ext_{G_4}^*\left(\omega_{4,4},\pi\right)_{sm} =0
\end{equation} 
for all irreducible representation $\pi$ of $G_4$ and all $*>1$. Indeed, by a standard degree shifting argument, to show (\ref{E:G4Ext2vanish}), it suffices to show that: for any irreducible representation $\pi$ of $G_4$, there exists a finite length representation $\Pi$ of $G_4$, such that $\pi$ is a quotient of $\Pi$, and
\[
    \Ext_{G_4}^2\left(\omega_{4,4}, \Pi\right)_{sm} = 0.
\] 
We now check this when $\pi$ belongs to the principal Bernstein block (i.e. the Bernstein block corresponding to the cuspidal component $\left[T_4, \mathbbm{1}_{T_4}\right]$, where $T_4$ is the diagonal torus of $G_4$, and $\mathbbm{1}_{T_4}$ is the trivial representation of $T_4$). For other Bernstein blocks, the proof is similar (and easier). So we can take $\Pi= \chi_1\times\chi_2\times\chi_3\times\chi_4$ to be a principal series representation such that there is a surjection
\[
    \chi_1\times\chi_2\times\chi_3\times\chi_4 \twoheadrightarrow \pi.
\]
Similar to the proof of Proposition \ref{P:BoundExtVanishingDegree}, by using the Frobenius reciprocity and the Kudla's filtration we obtain an exact sequence
\begin{flushleft}
$\cdots \lra \Ext_{G_1\times G_3}^2\left(|\cdot|^{\frac{1}{2}}\boxtimes\omega_{3,4},\chi_1\boxtimes\left(\chi_2\times\chi_3\times\chi_4\right)\right)_{sm} $
\end{flushleft}
\begin{flushright}
$\lra\Ext_{G_4}^2\left(\omega_{4,4},\chi_1\times\chi_2\times\chi_3\times\chi_4\right)_{sm} \lra \chi_1\times\Ext_{G_3}^2\left(\omega_{3,3}, \chi_2\times\chi_3\times\chi_4\right)_{sm} \lra \cdots$   
\end{flushright}
It follows from Proposition \ref{P:BoundExtVanishingDegree} that
\[
    \Ext_{G_3}^2\left(\omega_{3,3}, \chi_2\times\chi_3\times\chi_4\right)_{sm} =0.
\]
On the other hand, by the K\"unneth formula Theorem \ref{T:Kunneth} and also Proposition \ref{P:BoundExtVanishingDegree} we have
\begin{flushleft}
    $\Ext_{G_1\times G_3}^2\left(|\cdot|^{\frac{1}{2}}\boxtimes\omega_{3,4},\chi_1\boxtimes\left(\chi_2\times\chi_3\times\chi_4\right)\right)_{sm}$
\end{flushleft}
\begin{flushright}
    $\simeq \Ext^1_{G_1}\left(|\cdot|^{\frac{1}{2}},\chi_1\right)\otimes\Ext_{G_3}^1\left(\omega_{3,4}, \chi_2\times\chi_3\times\chi_4\right)_{sm}$.
\end{flushright}
There are three cases.

\vskip 5pt

$\bullet$ \textit{Case 1}: if $\chi_1\not\simeq|\cdot|^{\frac{1}{2}}$, then by Lemma \ref{L:ExtSteinberg}, we have
\[
    \Ext^1_{G_1}\left(|\cdot|^{\frac{1}{2}},\chi_1\right)=0.
\]

\vskip 5pt

$\bullet$ \textit{Case 2}: if $\chi_1\simeq|\cdot|^{\frac{1}{2}}$ but $\chi_2\not\simeq |\cdot|^1$, then by Corollary \ref{C:KudlaIndCompatible}, we have
\[
    \Ext_{G_3}^1\left(\omega_{3,4}, \chi_2\times\chi_3\times\chi_4\right)_{sm} \simeq \chi_2 \times \Ext_{G_2}^1\left(\omega_{2,3}, \chi_3\times\chi_4\right)_{sm}.
\]
Thanks to Theorem \ref{T:Main-3}, we have $\Ext_{G_2}^1\left(\omega_{2,3}, \chi_3\times\chi_4\right)_{sm}=0$, since $\omega_{2,3}$ is projective.

\vskip 5pt

$\bullet$ \textit{Case 3}: if $\chi_1\simeq|\cdot|^{\frac{1}{2}}$ and $\chi_2\simeq |\cdot|^1$, then 
\[
    \chi_1\times\chi_2\simeq \chi_2\times\chi_1.
\]
Therefore we can exchange the role of $\chi_1$ and $\chi_2$ by exchanging the induction order in $\Pi$, and reduce this case to Case 1 above.

\vskip 5pt

In conclusion, we always have
\[
    \Ext_{G_1\times G_3}^2\left(|\cdot|^{\frac{1}{2}}\boxtimes\omega_{3,4},\chi_1\boxtimes\left(\chi_2\times\chi_3\times\chi_4\right)\right)_{sm} = 0.
\]
This implies that $\Ext_{G_4}^2\left(\omega_{4,4},\Pi\right)_{sm}=0$ as desired. Combining Proposition \ref{P:BoundExtVanishingDegree} with these computations, we get the following consequence.

\begin{Cor}
Suppose that $n\leq 4$, and $n\leq m$. Then for any open compact subgroup $K$ of $H_m$, we have
\[
    {\rm proj.}\dim \omega_{n,m}^K \leq 1.
\]   
\end{Cor}

\vskip 5pt

\begin{Rmk}
It is worth noting that, among all examples we have computed, there is no irreducible representation $\pi$ of $G_n$ such that $\Ext_{G_n}^2\left(\omega_{n,m},\pi\right)_{sm}\neq 0$. It is interesting to ask: do we always have
\[
    {\rm proj.}\dim \omega_{n,m}^K \leq 1?
\] 
Also, throughout this subsection we only consider the ``level $K$'' part $\omega_{n,m}^K$. One can also ask what is the relation between ${\rm proj.}\dim \omega_{n,m}$ and ${\rm proj.}\dim \omega_{n,m}^K$. For example, do we have
\[
    {\rm proj.}\dim \omega_{n,m} = \sup_K\left\{{\rm proj.}\dim \omega_{n,m}^K\right\}?
\]
We do not know the answer.
\end{Rmk}


\vskip 10pt


\section{Further evidence for the speculation: low rank examples} 
\label{sec:speculation_ext_spaces_versus_l_functions}

In this section we provide further evidences for our speculations Question \ref{Q:ExtcontrolLfunction}. 
The main goal is to prove Proposition \ref{P:leq3computation}. Before we begin the proof, we briefly describe the idea. Suppose that $n=m$, and $\pi$ is an irreducible representation of $G_n$ such that both $L\left(s,\pi\right)$ and $L\left(s,\pi^\vee\right)$ have poles at $s=1/2$. If we can show that $\Hom_{G_n}\left(\omega_{n,n},\pi\right)_{sm}$ is not irreducible, then by Theorem \ref{T:APS-EP-fomula} there must exist some $*>0$ such that
\[
    \Ext_{G_n}^*\left(\omega_{n,n},\pi\right)_{sm} \neq 0,
\]
which completes the proof. Our idea to show the reducibility of $\Hom_{G_n}\left(\omega_{n,n},\pi\right)_{sm}$ is as follows. Suppose that we can find a \textit{length two} representation $\Pi$, such that: 
\begin{itemize}
    \item $\Pi$ fits into a short exact sequence
    \[
        0 \lra \pi \lra \Pi \lra \pi' \lra 0,
    \]
    which is \textit{non-split}, and $\pi\not\simeq \pi'$;

    \vskip 5pt

    \item the assertions in \cite[Thm. 3.3(1)-(3)]{MR342495} hold for $\Pi$, and $L\left(s,\Pi\right) = L\left(s,\pi\right)$;

    \vskip 5pt

    \item the ratio of the L-functions
    $L\left(s,\pi'\right)/L\left(s,\pi\right)$
    has a \textit{zero} at $s=1/2$.
\end{itemize}
\vskip 5pt
Given such a representation $\Pi$, the zeta integral of $\Pi$ then yields a non-zero map
\[
    \frac{Z\left(\frac{1}{2},-,-\right)}{L\left(\frac{1}{2},\Pi\right)}: \Pi \lra \Hom_{G_n}\left(\omega_{n,n}, \Pi\right)_{sm}.
\]
Since we have assumed that $L\left(s,\pi'\right)/L\left(s,\pi\right)$ has a zero at $s=1/2$, easy to check this map factors through the inclusion
\[
    \Hom_{G_n}\left(\omega_{n,n}, \pi\right)_{sm} \hookrightarrow \Hom_{G_n}\left(\omega_{n,n}, \Pi\right)_{sm},
\]
Hence we get a non-zero map
\[
    \Pi \lra \Hom_{G_n}\left(\omega_{n,n}, \pi\right)_{sm}.
\]
Note that by the Howe duality Theorem \ref{T:HoweDuality}, this map must be injective. Therefore the Hom space $\Hom_{G_n}\left(\omega_{n,n},\pi\right)_{sm}$ is not irreducible. As an example, when $n=m=2$ and $\pi=\mathbbm{1}_2$ is the trivial representation, $\Pi = |\cdot|^{-\frac{1}{2}}\times|\cdot|^{\frac{1}{2}}$ satisfies these requirements. 

\vskip 5pt

Now back to the proof of Proposition \ref{P:leq3computation}. As explicated in Section \ref{sub:speculation_i}, it only remains to check the case $n=m=3$. There are several subcases, we consider them separately.

\vskip 5pt

$\bullet$ \textit{Subcase 1}: $\pi\simeq|\det|^{\pm\frac{1}{2}}$ is a character. We have already proved in Theorem \ref{T:Main-2} that $\Hom_{G_3}\left(\omega_{3,3},\pi\right)_{sm}$ is not irreducible. 

\vskip 5pt

$\bullet$ \textit{Subcase 2}: $\pi\simeq \chi\times\mathbbm{1}_2$ for some character $\chi\in\Irr\left(G_1\right)\backslash\left\{|\cdot|^{\pm\frac{3}{2}}\right\}$. By Corollary \ref{C:ExtvsMVW}, without loss of generality we may assume that $\chi\not\simeq|\cdot|^{\frac{1}{2}}$. Applying Corollary \ref{C:KudlaIndCompatible}, we get
\[
    \Hom_{G_3}\left(\omega_{3,3},\pi\right)_{sm} \simeq \chi \times \Hom_{G_2}\left(\omega_{2,2},\mathbbm{1}_2\right)_{sm}.
\]
Since $\Hom_{G_2}\left(\omega_{2,2},\mathbbm{1}_2\right)_{sm}$ is not irreducible, $\Hom_{G_3}\left(\omega_{3,3},\pi\right)_{sm}$ is also not irreducible.

\vskip 5pt

$\bullet$ \textit{Subcase 3}: $\pi\simeq LQ\left(\St\left(\Delta_+\right)\times |\cdot|^{-\frac{1}{2}}\right)$ or $\pi \simeq LQ\left(|\cdot|^{\frac{1}{2}}\times\St\left(\Delta_-\right)\right)$, where the segment $\Delta_+=\left[\frac{3}{2},\frac{1}{2}\right]_{\mathbbm{1}_1}$ and $\Delta_-=\left[-\frac{1}{2},-\frac{3}{2}\right]_{\mathbbm{1}_1}$.
In this subcase we can simply take
\[
    \Pi = |\cdot|^{-\frac{1}{2}}\times\St\left(\Delta_+\right) \quad \textit{or} \quad \Pi= \St\left(\Delta_-\right) \times |\cdot|^{\frac{1}{2}},
\]
depending on the representation $\pi$. Then by Theorem \ref{T:IndSeg}, $\Pi$ satisfies desired properties. Therefore we have an injection
\[
    \Pi\hookrightarrow \Hom_{G_3}\left(\omega_{3,3},\pi\right)_{sm},
\]
and it follows that $\Hom_{G_3}\left(\omega_{3,3},\pi\right)_{sm}$ is not irreducible.

\vskip 5pt

Hence we have verified in all subcases that $\Hom_{G_3}\left(\omega_{3,3},\pi\right)_{sm}$ is not irreducible. In conclusion, Proposition \ref{P:leq3computation} holds.

\vskip 15pt

\bibliographystyle{alpha}
\bibliography{LfunctionNHThetaRef}

\begin{thebibliography}{GGP20}

\bibitem[AG17]{MR3714507}
Hiraku Atobe and Wee~Teck Gan.
\newblock Local theta correspondence of tempered representations and
  {L}anglands parameters.
\newblock {\em Invent. Math.}, 210(2):341--415, 2017.

\bibitem[APS17]{MR3753906}
Jeffrey~D. Adams, Dipendra Prasad, and Gordan Savin.
\newblock Euler-{P}oincar\'{e} characteristic for the oscillator
  representation.
\newblock In {\em Representation theory, number theory, and invariant theory},
  volume 323 of {\em Progr. Math.}, pages 1--22. Birkh\"{a}user/Springer, Cham,
  2017.

\bibitem[AS20]{MR4054816}
Avraham Aizenbud and Eitan Sayag.
\newblock Homological multiplicities in representation theory of {$p$}-adic
  groups.
\newblock {\em Math. Z.}, 294(1-2):451--469, 2020.

\bibitem[Ber92]{BNote}
Joseph Bernstein.
\newblock Representations of {$p$}-{A}dic groups.
\newblock \textit{Online note, available at
  }\url{https://personal.math.ubc.ca/~cass/research/pdf/bernstein.pdf}, 1992.

\bibitem[Cha21]{MR4270667}
Kei~Yuen Chan.
\newblock Homological branching law for {$({\rm GL}_{n+1}(F), {\rm GL}_n(F))$}:
  projectivity and indecomposability.
\newblock {\em Invent. Math.}, 225(1):299--345, 2021.

\bibitem[CS19]{MR3910471}
Kei~Yuen Chan and Gordan Savin.
\newblock Bernstein-{Z}elevinsky derivatives: a {H}ecke algebra approach.
\newblock {\em Int. Math. Res. Not. IMRN}, (3):731--760, 2019.

\bibitem[CS21a]{MR4291425}
Kei~Yuen Chan and Gordan Savin.
\newblock A vanishing {E}xt-branching theorem for {$({\rm GL}_{n+1}(F),{\rm
  GL}_n(F))$}.
\newblock {\em Duke Math. J.}, 170(10):2237--2261, 2021.

\bibitem[CS21b]{MR4211018}
Yangyang Chen and Binyong Sun.
\newblock Schwartz homologies of representations of almost linear {N}ash
  groups.
\newblock {\em J. Funct. Anal.}, 280(7):Paper No. 108817, 50, 2021.

\bibitem[FSX18]{FSX}
Yingjue Fang, Binyong Sun, and Huajian Xue.
\newblock Godement-{J}acquet {L}-functions and full theta lifts.
\newblock {\em Math. Z.}, 289(1-2):593--604, 2018.

\bibitem[Gan19]{MR3930015}
Wee~Teck Gan.
\newblock Periods and theta correspondence.
\newblock In {\em Representations of reductive groups}, volume 101 of {\em
  Proc. Sympos. Pure Math.}, pages 113--132. Amer. Math. Soc., Providence, RI,
  2019.

\bibitem[GGP20]{MR4190046}
Wee~Teck Gan, Benedict~H. Gross, and Dipendra Prasad.
\newblock Branching laws for classical groups: the non-tempered case.
\newblock {\em Compos. Math.}, 156(11):2298--2367, 2020.

\bibitem[GJ72]{MR342495}
Roger Godement and Herv\'e Jacquet.
\newblock {\em Zeta functions of simple algebras}, volume Vol. 260 of {\em
  Lecture Notes in Mathematics}.
\newblock Springer-Verlag, Berlin-New York, 1972.

\bibitem[How89]{MR985172}
Roger Howe.
\newblock Transcending classical invariant theory.
\newblock {\em J. Amer. Math. Soc.}, 2(3):535--552, 1989.

\bibitem[HS17]{MR3623236}
Jiuzu Hong and Binyong Sun.
\newblock Generalized semi-invariant distributions on {$p$}-adic spaces.
\newblock {\em Math. Ann.}, 367(3-4):1727--1776, 2017.

\bibitem[LM16]{MR3573961}
Erez Lapid and Alberto M\'inguez.
\newblock On parabolic induction on inner forms of the general linear group
  over a non-archimedean local field.
\newblock {\em Selecta Math. (N.S.)}, 22(4):2347--2400, 2016.

\bibitem[M{\'i}n08]{MR2504432}
Alberto M{\'i}nguez.
\newblock Correspondance de {H}owe explicite: paires duales de type {II}.
\newblock {\em Ann. Sci. \'Ec. Norm. Sup\'er. (4)}, 41(5):717--741, 2008.

\bibitem[NP20]{Duality}
Madhav Nori and Dipendra Prasad.
\newblock On a duality theorem of {S}chneider-{S}tuhler.
\newblock {\em J. Reine Angew. Math.}, 762:261--280, 2020.

\bibitem[Orl05]{MR2173717}
Sascha Orlik.
\newblock On extensions of generalized {S}teinberg representations.
\newblock {\em J. Algebra}, 293(2):611--630, 2005.

\bibitem[Pra18]{MR3966813}
Dipendra Prasad.
\newblock Ext-analogues of branching laws.
\newblock In {\em Proceedings of the {I}nternational {C}ongress of
  {M}athematicians---{R}io de {J}aneiro 2018. {V}ol. {II}. {I}nvited lectures},
  pages 1367--1392. World Sci. Publ., Hackensack, NJ, 2018.

\bibitem[Pra23]{prasad2023homological}
Dipendra Prasad.
\newblock Homological aspects of branching laws.
\newblock {\em arXiv preprint arXiv:2302.03492}, 2023.

\bibitem[Qad25]{MR4846726}
Mohammed~Saad Qadri.
\newblock Non-tempered {E}xt branching laws for the {$p$}-adic general linear
  group.
\newblock {\em Int. Math. Res. Not. IMRN}, (1):Paper No. rnae263, 25, 2025.

\bibitem[Ren10]{MR2567785}
David Renard.
\newblock {\em Repr\'esentations des groupes r\'eductifs {$p$}-adiques},
  volume~17 of {\em Cours Sp\'ecialis\'es [Specialized Courses]}.
\newblock Soci\'et\'e{} Math\'ematique de France, Paris, 2010.

\bibitem[Sil78]{MR507262}
Allan~J. Silberger.
\newblock The {L}anglands quotient theorem for {$p$}-adic groups.
\newblock {\em Math. Ann.}, 236(2):95--104, 1978.

\bibitem[Xue18]{MR3673813}
Huajian Xue.
\newblock Homogeneous distributions on finite dimensional vector spaces.
\newblock {\em J. Lie Theory}, 28(1):33--41, 2018.

\bibitem[Xue23]{xue2023full}
Huajian Xue.
\newblock Full theta lifting for type {II} reductive dual pairs.
\newblock {\em arXiv preprint arXiv:2312.13073}, 2023.

\bibitem[Zel80]{ZeleII}
A.~V. Zelevinsky.
\newblock Induced representations of reductive {${\frak p}$}-adic groups. {II}.
  {O}n irreducible representations of {${\rm GL}(n)$}.
\newblock {\em Ann. Sci. \'Ecole Norm. Sup. (4)}, 13(2):165--210, 1980.

\end{thebibliography}

\end{document}